\documentclass[12pt]{amsart}

\usepackage{latexsym, amssymb, amsmath}
\usepackage{amsfonts, graphicx}
\usepackage{graphicx,color}
\newcommand{\be}{\begin{equation}}
\newcommand{\ee}{\end{equation}}
\newcommand{\beq}{\begin{eqnarray}}
\newcommand{\eeq}{\end{eqnarray}}

\newtheorem{thm}{Theorem}[section]

\newtheorem{lma}{Lemma}[section]
\newtheorem{prop}{Proposition}[section]
\newtheorem{cor}{Corollary}[section]

\newtheorem{sublma}{Sublemma}[section]

\theoremstyle{remark}
\newtheorem{rem}{Remark}[section]
\numberwithin{equation}{section}

\def\dps{\displaystyle}

\def\be{\begin{equation}}
\def\ee{\end{equation}}
\def\bee{\begin{equation*}}
\def\eee{\end{equation*}}
\def\ol{\overline}
\def\lf{\left}
\def\ri{\right}

\def\K{K\"ahler }

\def\KE{K\"ahler-Einstein }
\def\KR{K\"ahler-Ricci }
\def\Ric{\text{\rm Ric}}
\def\Rm{\text{\rm Rm}}

\def\wt{\widetilde}

\def\p{\partial}
\def\ddbar{\partial\bar\partial}

\def\ol{\overline}
\def\heat{\lf(\frac{\p}{\p t}-\Delta\ri)}
\def\tr{\operatorname{tr}}
\def\sheat{\lf(\frac{\p}{\p s}-\wt\Delta\ri)}

\def\e{\epsilon}
\def\a{{\alpha}}
\def\b{{\beta}}

\def\ijb{{i\bar{j}}}
\def\R{\mathbb{R}}
\def\C{\mathbb{C}}

\def\ddb{\sqrt{-1}\partial\bar\partial}

\def\dps{\displaystyle}

\def\be{\begin{equation}}
\def\ee{\end{equation}}
\def\bee{\begin{equation*}}
\def\eee{\end{equation*}}
\def\ol{\overline}
\def\lf{\left}
\def\ri{\right}

\def\K{K\"ahler }
\def\KR{K\"ahler-Ricci }
\def\Ric{\text{\rm Ric}}
\def\Rm{\text{\rm Rm}}

\def\wt{\widetilde}

\def\p{\partial}
\def\ddbar{\partial\bar\partial}

\def\ol{\overline}
\def\heat{\lf(\frac{\p}{\p t}-\Delta\ri)}
\def\tr{\operatorname{tr}}

\def\e{\epsilon}
\def\a{{\alpha}}
\def\b{{\beta}}

\def\ijb{{i\bar{j}}}

\def\ii{\sqrt{-1}}
\def\R{\mathbb{R}}
\def\C{\mathbb{C}}

\def\n{\nabla}
\def\ppt{ \frac{\p}{\p t}}

\newcounter{mnotecount}[section]

\begin{document}

\title[]
{Instantaneously  complete Chern-Ricci flow and K\"ahler-Einstein metrics }

 \author{Shaochuang Huang$^1$}
\address[Shaochuang Huang]{Yau Mathematical Sciences Center, Tsinghua University, Beijing, China.}
\email{schuang@mail.tsinghua.edu.cn}
\thanks{$^1$Research   partially supported by China Postdoctoral Science Foundation \#2017T100059}

 \author{Man-Chun Lee}
\address[Man-Chun Lee]{Department of
 Mathematics, University of British Columbia, Canada}
\email{mclee@math.ubc.ca}

\author{Luen-Fai Tam$^2$}
\address[Luen-Fai Tam]{The Institute of Mathematical Sciences and Department of
 Mathematics, The Chinese University of Hong Kong, Shatin, Hong Kong, China.}
 \email{lftam@math.cuhk.edu.hk}
\thanks{$^2$Research  partially supported by Hong Kong RGC General Research Fund \#CUHK 14301517}

\renewcommand{\subjclassname}{
  \textup{2010} Mathematics Subject Classification}

\date{February 2019}

\begin{abstract}
In this work, we obtain some existence results of {\it Chern-Ricci Flows} and the corresponding {\it Potential Flows} on complex manifolds with possibly incomplete initial data. We  discuss the behaviour of the solution as $t\to 0$. These results can be viewed as a generalization of an existence result of Ricci flow by Giesen and Topping for surfaces of hyperbolic type to higher dimensions in certain sense. On the other hand, we also discuss the long time behaviour of the solution and obtain some sufficient conditions for the existence of \KE metric on complete non-compact Hermitian manifolds, which generalizes the work of Lott-Zhang and Tosatti-Weinkove to complete non-compact Hermitian manifolds with possibly unbounded curvature.

\end{abstract}


\maketitle

\markboth{Shaochuang Huang, Man-Chun Lee and Luen-Fai Tam}{Instantaneously  complete Chern-Ricci flow and K\"ahler-Einstein metrics }

\section{Introduction}

In this work, we will discuss conditions on the existence of {\it Chern-Ricci Flows} and the corresponding {\it Potential Flows} on complex manifolds with possibly incomplete initial data. The flows will be described later. We will also discuss conditions on long-time existence and convergence to K\"ahler-Einstein metrics.

 We begin with the definitions of Chern-Ricci flow and the corresponding potential flow. Let $M^n$ be a complex manifold with complex dimension $n$. Let $h$ be a Hermitian metric on $M$ and let $\theta_0$ be the \K form of $h$:
$$
\theta_0=\ii h_\ijb dz^i\wedge d\bar z^j
$$
where $h=h_\ijb dz^i\otimes d\bar z^j$ in local holomorphic coordinates. {In this work, Einstein summation convention is enforced.}

 In general, suppose $\omega$ is a real (1,1) form on $M$, if $\omega=\ii g_\ijb dz^i\wedge d\bar z^j$ in local holomorphic coordinates then   the corresponding Hermitian form $g$ is given by
$$
g=g_\ijb dz^i\otimes d\bar z^j.
$$
 In case $\omega$ is only nonnegative, we   call $g$ to be the Hermitian form of $\omega$ and $\omega$ is still called the \K form of $g$.

Now if $(M^n,h)$ is a Hermitian manifold with \K form $\theta_0$, let $\nabla$ be the Chern connection $\nabla $ of $h$ and $\Ric(h)$ be the Chern-Ricci tensor of $h$ (or the first Ricci curvature).  In holomorphic local coordinates such that $h=h_\ijb dz^i\otimes d\bar z^j$, the Chern Ricci form is given by
$$
\Ric(h)=-\ii \p\bar\p \log \det(h_\ijb).
$$
For the basic facts on Chern connection and Chern curvature, we refer readers to \cite[section 2]{ TosattiWeinkove2015}, see also \cite[Appendix A]{Lee-Tam} for example.

Let $\omega_0$ be another  nonnegative real (1,1) form on $M$. Define
\be\label{e-alpha}
\a:=-\Ric(\theta_0)+e^{-t}\lf(\Ric(\theta_0)+\omega_0\ri)
\ee
where $\Ric(\theta_0)$ is the Chern-Ricci curvature of $h$. We want to study the following  parabolic complex Monge-Amp\`ere equation:
\be\label{e-MP-1}
\left\{
  \begin{array}{ll}
   {\dps \frac{\p u}{\p t}}&=\displaystyle{\log\lf(\frac{(\a+\ii\ddbar u)^n}{\theta_0^n}\ri)}-u\ \ \text{in $M\times(0,S]$} \\
    u(0)&=0
  \end{array}
\right.
\ee
so that $\a+\ii\ddbar u>0$ for $t>0$.  When $M$ is compact and $\omega_0=\theta_0$ is smooth metric, it was first studied by Gill in \cite{Gill}. Here we are interested in the case when $\omega_0$ is possibly an incomplete metric on a complete non-compact Hermitian manifold $(M,h)$. Following \cite{Lott-Zhang}, \eqref{e-MP-1} will be called the {\it potential flow} of the following    normalized Chern-Ricci flow:
\be\label{e-NKRF}
\left\{
  \begin{array}{ll}
   {\dps \frac{\p}{\p t}\omega(t)} &=   -\Ric(\omega(t))-\omega(t); \\
    \omega(0)&=  \omega_0.
  \end{array}
\right.
\ee
It is easy to see that the normalized Chern-Ricci flow will coincide with the normalized \KR flow if $\omega_0$ is K\"ahler. It is well-known that if $\omega_0$ is a Hermitian metric and $\omega(t)$ is Hermitian and a solution to \eqref{e-NKRF} which is smooth up to $t=0$, then
\be\label{e-potential}
u(t)=e^{-t}\int_0^te^s\log \frac{(\omega (s))^n}{\theta_0^n}ds.
\ee
satisfies \eqref{e-MP-1}. Moreover, $u(t)\to0$ in $C^\infty$ norm in any compact set as $t\to0$. On the other hand, if $u$ is a solution to \eqref{e-MP-1} so that $\a+\ii\ddbar u>0$ for $t>0$, then
\be\label{e-potential-1}
\omega(t)=\a+\ii\ddbar u
\ee is a solution to \eqref{e-NKRF} on $M\times(0,S]$. However, even if we  know $u(t)\to0$ as $t\to0$  uniformly on $M$, it is still unclear that $\omega(t)\to\omega_0$ in  general.

 The first motivation is to study Ricci flows starting from metrics which are possibly incomplete and with unbounded curvature. In complex dimension one, the existence of Ricci flow starting from an arbitrary metric has been studied in details by Giesen and Topping \cite{GiesenTopping-1,GiesenTopping-2, GiesenTopping,Topping}.  In particular, the following was proved in \cite{GiesenTopping}: {\it If a surface admits a complete metric $H$ with constant negative curvature, then any initial data which may be incomplete can be deformed through the normalized Ricci flow for long time and converges to $H$. Moreover, the solution is instantaneously complete for $t>0$.}  In higher dimensions, recently it is proved by Ge-Lin-Shen \cite{Ge-Lin-Shen} that on a complete non-compact \K manifold $(M,h)$ with $\Ric(h)\leq -h$ and bounded curvature, if $\omega_0$ is a \K metric,  not necessarily complete, but with bounded $C^k$ norm with respect to $h$ for $k\ge 0$, then   \eqref{e-NKRF} has a long time solution which  converges to the unique K\"ahler-Einstein metric with negative scalar curvature, by solving \eqref{e-MP-1}. Moreover, the solution is instantaneously complete after it evolves.

Motivated by  the above mentioned  works, we first study the short time existence of the potential flow and the normalized Chern-Ricci flow. Our first result is the following:
\begin{thm}\label{main-instant-complete}
Let $(M^n,h)$ be a complete non-compact Hermitian manifold with complex dimension $n$. Suppose there is $K>0$ such that the following hold.
\begin{enumerate}
\item There is a {proper} exhaustion function $\rho(x)$ on $M$ such that
$$|\partial\rho|^2_h +|\ddb \rho|_h \leq K.$$
\item  $\mathrm{BK}_h\geq -K$;
\item The torsion of $h$, $T_h=\partial \omega_h$ satisfies
$$|T_h|^2_h +|\nabla^h_{\bar\partial} T_h |\leq K.$$
\end{enumerate}
Let $\omega_0$ be a nonnegative real (1,1) form with corresponding Hermitian form  $g_0$   on $M$ (possibly incomplete {or degenerate}) such that
\begin{enumerate}
\item[(a)] $g_0\le h$ and
$$|T_{g_0}|_h^2+|\nabla^h_{\bar\partial} T_{g_0}|_h+ |\nabla^{h}g_0|_h\leq K.$$

\item[(b)] There exist $f\in C^\infty(M)\cap L^\infty(M),\beta>0$ and $s>0$ so that  $$-\Ric(\theta_0)+e^{-s}(\omega_0+\Ric(\theta_0))+\ddb f\geq \beta \theta_0.$$

\end{enumerate}
Then \eqref{e-MP-1} has a solution on $M\times(0, s)$ so that $u(t)\to 0$ as $t\to0$ uniformly on $M$. Moreover, for any $0<s_0<s_1<s$, $\omega(t)=\a+\ii\ddbar u$ is the \K form of a complete Hermitian metric which is uniformly equivalent to $h$ on $M\times[s_0, s_1]$. In particular, $g(t)$ is complete for $t>0$.
\end{thm}
Here $\mathrm{BK}_h\geq -K$ means that for any unitary frame $\{e_k\}$ of $h$, we have $R(h)_{i\bar ij\bar j}\geq -K$ for all $i,j$.
\begin{rem}
It is well-known that when $(M,h)$ is K\"ahler with bounded curvature, then condition (1) will be satisfied, \cite{Shi1989,Tam2010}. See also \cite{NiTam2013,Huang2018} for related results under various assumptions.
\end{rem}

Condition (b) was used in \cite{Lott-Zhang,TosattiWeinkove2015, Lee-Tam}   with $\omega_0$ replaced by $\theta_0$ and  is motivated as pointed out in \cite{Lott-Zhang} as follows. If we are considering  cohomological class instead, in case that  $\omega(t)$ is closed, then \eqref{e-NKRF} is:
$$
\partial_t[\omega(t)]=-[\Ric(\omega(t)]-[\omega(t)]
$$
and so
$$
[\omega(t)]=-(1-e^{-t})[\Ric(\theta_0)]+e^{-t}[\omega_0].
$$
Condition (b) is used to guarantee that $\omega(t)>0$. In our case $\omega_0,\theta_0, \omega(t)$ may not be closed and $\omega_0$ may degenerate. These may cause some difficulties. Indeed, the result is analogous to running \KR flow from a rough initial data. When $M$ is compact, the potential flow from a rough initial data had already been studied by several authors, see for example \cite{BG2013,SongTian2017,To2017} and the references therein.

On the other hand, a solution of \eqref{e-MP-1} gives rise to a solution of  \eqref{e-NKRF} when $t>0$. It is rather delicate to see if the corresponding solution of \eqref{e-NKRF} will attain the initial Hermitian form $\omega_0$. In this respect, we will prove the following:
\begin{thm}\label{t-initial-Kahler-1}
With the same notation  and assumptions as in Theorem \ref{main-instant-complete}. Let $\omega(t)$ be the solution of \eqref{e-NKRF} obtained in the theorem.  If in addition $h$ is \K and $d\omega_0=0$. Let $U=\{\omega_0>0\}$.  Then $\omega(t)\rightarrow \omega_0$ in $C^\infty(U)$ as $t\rightarrow 0$, {uniformly on compact subsets of $U$}. 
\end{thm}
 We should remark that if in addition $h$ has bounded curvature, then the theorem follows easily from pseudo-locality. The theorem can be applied to the cases studied in \cite{Ge-Lin-Shen} and to the case that $-\Ric(h)\ge \b\theta_0$ outside a compact set $V$ and   $\omega_0>0$ on $V$ with $\omega_0$ and its first covariant derivative are bounded. In particular, when $\Omega$ is a bounded strictly pseudoconvex domain of another manifold $M$ with defining function $\varphi$, then the $\Omega$ with the metric $h_{i\bar j}=-\partial_i \partial_{\bar j}\log(-\varphi)$ will satisfy the above, see \cite[(1.22)]{ChengYau1982}.

Another motivation here is to study the existence of K\"ahler-Einstein metric with negative scalar curvature on complex manifolds using geometric flows.   In \cite{Aubin, Yau1978-2}, Aubin and Yau proved that if $M$ is a compact \K manifold with negative first Chern class $c_1(M)<0$, then it admits a unique  K\"ahler-Einstein metric with negative scalar curvature by studying the elliptic complex Monge-Amp\`ere equation. Later, Cao \cite{Cao} reproved the above result using the K\"ahler-Ricci flow by showing that one can deform a suitable initial \K metric through normalized \KR flow to the K\"ahler-Einstein metric.  Recently, Tosatti and Weinkove \cite{TosattiWeinkove2015}  proved that under the same condition that $c_1(M)<0$ on a compact complex manifold, the normalized Chern-Ricci flow \eqref{e-NKRF} with an arbitrary Hermitian initial metric also has a long time solution and converges to the K\"ahler-Einstein metric with negative scalar curvature.
In \cite{ChengYau1982}, Cheng and Yau proved that if $M$ is a complete non-compact \K manifold with Ricci curvature bounded above by a negative constant, injectivity radius bounded below by a positive constant and curvature tensor with its covariant derivatives are bounded, then $M$ admits a unique complete  K\"ahler-Einstein metric with negative scalar curvature.    In \cite{Chau04}, Chau used K\"ahler-Ricci flow to prove that if $(M, g)$ is a complete non-compact \K manifold with bounded curvature and $\Ric(g)+g=\ii\ddbar f $ for some smooth bounded function $f$, then it also admits a complete  K\"ahler-Einstein metric with negative scalar curvature.   Later, Lott and Zhang \cite{Lott-Zhang} generalized Chau's result by  assuming $$-\Ric(g)+\ii\ddbar f\ge\b g$$ for some smooth  function $f$ with bounded $k$th covariant derivatives for each $k\geq0$  and positive constant $\b$. In this work, we will generalize the results in \cite{Lott-Zhang,TosattiWeinkove2015} to complete non-compact Hermitian manifolds with possibly unbounded curvature.

For the long time existence and convergence, we will prove the following:
\begin{thm}\label{main-longtime}
Under the assumption of Theorem \ref{main-instant-complete}, if in addition, $$-\Ric(h)+\ddb f\geq \beta \theta_0$$ for some $f\in C^\infty(M)\cap L^\infty(M)$, $\beta>0$. Then the solution constructed from Theorem \ref{main-instant-complete} is a longtime solution and converges to a unique complete \K Einstein metric with negative scalar curvature on $M$.
\end{thm}
As a consequence, we see that if $h$ satisfies the conditions in the theorem, then $M$ supports a complete K\"ahler-Einstein metric with negative scalar curvature, generalizing the results in \cite{Lott-Zhang,TosattiWeinkove2015}.

The paper is organized as follows: In section 2, we will derive a priori estimates along the potential flow and apply it in section 3 to prove Theorem \ref{main-instant-complete}. Furthermore, we will study the short time behaviour of the constructed solution. In section 4, we will  prove the Theorem \ref{main-longtime} and discuss longtime behaviour for general \KR flow if the initial data satisfies some extra condition. In   Appendix A, we will collect some information about the relation between normalized Chern-Ricci flow and unnormalized one {together with  some useful differential inequalities. In Appendix B, we will state a maximum principle which will be used in this work.}

\section{a priori estimates for the potential flow}\label{s-aprior}

We will study the short time existence of the potential flow \eqref{e-MP-1} with $\omega_0$ only being assumed to be nonnegative. We need   some a priori estimates for the flow. In this section, we always assume the following:
\begin{enumerate}
\item There is a {proper} exhaustion function $\rho(x)$ on $M$ such that
$$|\partial\rho|^2_h +|\ddb \rho|_h \leq K.$$
\item $\mathrm{BK}_h\geq -K$.
\item The torsion of $h$, $T_h=\partial \omega_h$ satisfies
$$|T_h|^2_h +|\nabla^h_{\bar\partial} T_h |\leq K.$$
\end{enumerate}
Here $K$ is some positive constant.

On the other hand, let $\omega_0$ be a real (1,1) form with corresponding Hermitian form $g_0$. We always assume the following:
\begin{enumerate}
\item[(a)] $g_0\le h$ and
$$|T_{g_0}|_h^2+|\nabla^h_{\bar\partial} T_{g_0}|_h+ |\nabla^{h}g_0|_h\leq K.$$

\item[(b)] There exist $f\in C^\infty(M)\cap L^\infty(M),\beta>0$ and $s>0$ so that  $$-\Ric(\theta_0)+e^{-s}(\omega_0+\Ric(\theta_0))+\ddb f\geq \beta \theta_0.$$

\end{enumerate}
Note that if $g_0\le Ch$, then we can replace $h$ by $Ch$, then (b) is still satisfied with a possibly smaller $\beta$.

Since $g_0$ can be degenerate, we perturb $g_0$ in the following way: Let $1\ge \eta\ge 0$ be a smooth function on $\R$ such that $\eta(s)=1$ for $s\le 1$ and $\eta(s)=0$ for $s\ge 2$ so that $|\eta'|+|\eta''|\le c_1$, say. For $\e>0$ and $\rho_0>>1$, let $\eta_{0}(x)=\eta(\rho(x)/\rho_0)$. Consider the metric:

\be
\gamma_0=\gamma_0(\rho_0,\e)=\eta_0\omega_0+(1-\eta_0)\theta_0+\e\theta_0.
\ee

 Then
\begin{itemize}
  \item $\gamma_0$ is the \K form of a complete Hermitian metric, which is uniformly equivalent to $h$;
       \item   $\mathrm{BK}(\gamma_0 )\ge -C$ for some $C$ which may depend on $\rho_0, \e$;
            \item The torsion $|T_{\gamma_0} |_{\gamma_0}+|\nabla^{\gamma_0}_{\bar \partial} T_{\gamma_0}|_{\gamma_0}$ is uniformly bounded by a constant which may depend on $\rho_0, \e$.
\end{itemize}

We will obtain a short time existence for the potential flow starting with $\gamma_0$:

 \begin{lma}\label{l-perturbed-1}  \eqref{e-MP-1} has a solution $u(t)$  on $M\times[0, s)$ with $\a=-\Ric(\theta_0)+e^{-t}\lf(\Ric(\theta_0)+\gamma_0\ri)$ and $\omega(t)=\a+\ii\ddbar u$  such that $\omega(t)$ satisfies \eqref{e-NKRF} with initial data $\gamma_0$, where $\omega(t)$ is the \K form of $g(t)$. Moreover, $g(t)$ is uniformly equivalent to $h$ on $M\times[0, s_1]$ for all $s_1<s$.
\end{lma}
\begin{proof} By  the proof of \cite[Theorem 4.1]{Lee-Tam}, it is sufficient to prove that for any $0<s_1<s$,
$$
-\Ric(\gamma_0)+e^{-s_1}(\gamma_0+\Ric(\gamma_0))+\ii\ddbar f_1\ge \beta_1\gamma_0
$$
for some smooth bounded function $f_1$ and some   constant $\beta_1>0$. To simplify the notations, if $\eta, \zeta$ are real (1,1) forms, we write $\eta \succeq \zeta$ if $\eta+\ii\ddbar \phi\ge \zeta$ for some smooth and bounded function $\phi$. We compute:
\bee\begin{split}
-\Ric(\gamma_0)+e^{-s_1}(\gamma_0+\Ric(\gamma_0))
=&-(1-e^{-s_1})\Ric(\gamma_0)+e^{-s_1}\gamma_0\\
\succeq&-(1-e^{-s_1})\Ric(\theta_0)+e^{-s_1}\gamma_0\\
\succeq&\frac{1-e^{-s_1}}{1-e^{-s}}(\beta \theta_0-e^{-s}\omega_0)+e^{-s_1}\gamma_0  \\
\ge&\frac{1-e^{-s_1}}{1-e^{-s}} \beta \theta_0 \\
\ge& \beta_1\gamma_0\end{split}\eee
for some $\beta_1>0$ because $0<s_1<s$ and $\gamma_0\ge \omega_0$. Here we have used {condition (b) above}, the fact that $\gamma_0^n =\theta_0^ne^H$ for some smooth bounded function $H$ and the definition of Chern-Ricci curvature.

\end{proof}
Let $\omega(t)$ be the solution in the lemma and let $u(t)$ be the potential as in \eqref{e-potential}.
Since we want to prove that \eqref{e-MP-1} has a solution $u(t)$  on $M\times(0, s)$ with $\a=-\Ric(\theta_0)+e^{-t}\lf(\Ric(\theta_0)+\omega_0\ri)$ in next section, we need to obtain some uniform estimates of $u, \dot u$ and $\omega(t)$ which is independent of $\rho_0$ and $\e$. The estimates are more delicate because   the initial data $\omega_0$ maybe degenerate. For later applications,  we need to obtain estimates on $(0,1]$ and $[1,s)$ if $s>1$. Note that for fixed $\rho_0, \e$, $u(t)$ is smooth up to $t=0$. Moreover, $u, \dot u=:\frac{\p}{\p t}u$ are uniformly bounded on $M\times[0,s_1]$ for all $0<s_1<s$.

\subsection{a priori estimates for $u$ and $\dot u$}\label{ss-uudot}

We first give estimates for upper bound of $u$ and $\dot u$.

\begin{lma}\label{l-uudot-upper-1} There is a constant $C$ depending only on $n$ and $K$ such that
$$
 u\le C\min\{t,1\}, \ \  \dot u\le  \frac{Ct}{e^t-1}
$$
on $M\times[0, s)$, provided $0<\e<1$.
\end{lma}
\begin{proof}The proofs here follow almost verbatim from the \K case \cite{TianZhang2006}, but we include brief arguments for the reader's convenience. For notational convenience, we use $\Delta=g^{i\bar j} \partial_i \partial_{\bar j}$ to denote the Chern Laplacian associated to $g(t)$.
Since $-\Ric(\theta_0)=\omega(t)-e^{-t}(\Ric(\theta_0)+\gamma_0)-\ii\ddbar u$
 by \eqref{e-potential-1}, we have
 \be\label{e-udot-1}
 \begin{split}
 \heat  (e^t\dot u)=&e^t\dot u-e^t \tr_{\omega}\Ric(\theta_0)-e^t\heat u-n e^t\\
 =&e^t\tr_\omega \lf(- \Ric(\theta_0)+\ii\ddbar u\ri)-ne^t\\
 =&e^t\tr_\omega\lf(\omega-e^{-t}(\Ric(\theta_0)+\gamma_0)\ri)-ne^t\\
 =&  -\tr_\omega (\Ric(\theta_0)+\gamma_0)\\
 =&\heat (\dot u+u)+n -\tr_\omega(\gamma_0).
 \end{split}
 \ee

Hence
\bee
\heat(\dot u+u+nt-e^t\dot u)=\tr_\omega\gamma_0\ge0.
\eee
At $t=0$, $\dot u+u+nt-e^t\dot u=0$. By maximum principle Lemma \ref{max}, we have
\be\label{e-uudot-upper-1}
(e^t-1)\dot u\le nt+u.
\ee

Next consider
\bee
F=u-At-\kappa\rho
\eee
on $M\times[0, s_1]$ for any fixed $s_1<s$. Here $\kappa>0$ is a constant.
Suppose $\sup\limits_{M\times[0, s_1]}F>0$, then there exists $(x_0, t_0)\in M\times(0, s_1]$ such that $F\leq F(x_0, t_0)$ on $M\times[0, s_1]$, and at this point,

\bee\begin{split}
0\leq& \dot u -A=\log \lf(\frac{\omega^n(t)}{\theta_0^n}\ri)-u-A. \end{split} \eee
Also, $\ii\ddbar u\le \kappa\ii\ddbar \rho\le \kappa K\theta_0$. Hence at $(x_0,t_0)$,
\bee
\begin{split}
\omega(t)=&-\Ric(\theta_0)+e^{-t}(\Ric(\theta_0)+\gamma_0)+\ii\ddbar u\\
\le&(-1+e^{-t})\Ric(\theta_0)+e^{-t}\gamma_0+\kappa K\theta_0\\
\le &(L+2+\kappa K)\theta_0,
\end{split}
\eee
here $\Ric(\theta_0)\ge -L(n, K)\theta_0$. Hence  at $(x_0,t_0)$ we have
\bee
\begin{split}
u\le & n\log(L+2+\kappa K)-A\\
\le &0
\end{split}
\eee
if $A=n\log(L+2)+1$ and $\kappa>0$ is small enough. Hence $F(x_0,t_0)<0$. This is a contradiction. Hence $F\le 0$ on $M\times[0, s_1]$ provided $A=A(n,K)=n\log(L+2)+1$ and we have
\be\label{e-uudot-upper-2}
u\le At
\ee
by letting $\kappa\to0$. Combining this with \eqref{e-uudot-upper-1}, we conclude that
$$
\dot u\le \frac{(A+n)t}{e^t-1}.
$$
Combining this with \eqref{e-uudot-upper-2}, we conclude that $u\le C$ for some constant $C$ depending only on $n, K$. Since $s_1$ is arbitrary, we complete the proof of Lemma \ref{l-uudot-upper-1}.
\end{proof}

Next, we will estimate the lower bound of $u$ and $\dot u$.
\begin{lma}\label{l-all-u}\begin{enumerate}
\item[(i)] $u(x,t)\geq - \frac{C}{1-e^{-s}} t+nt\log(1-e^{-t})$ on  $M\times[0, s)$ for some constant $C>0$ depending only on $ n, \beta, K, ||f||_\infty$.
              \item [(ii)] For $0<s_1\leq 1$ and $s_1<s$,
             \bee
             \dot u+u\ge\frac1{1-e^{s_1-s}}\lf(n\log t-\frac{C}{1-e^{-s}}\ri)
             \eee
             some constant $C>0$ depending only on $ n, \beta, K, ||f||_\infty$
             on $M\times(0, s_1]$.
             \item [(iii)]  For $0<s_1\leq 1$ and $s_1<s$,
 $$\dot u+u\geq -C$$
             on $M\times[0, s_1]$ for some constant $C>$ depending only on
             $ n, \beta$, $K, ||f||_\infty, s_1, s$ and $\e$.
             \item [(iv)] Suppose $s>1$, then for $1<s_1<s$,
             $$\dot u+u\ge  -\frac{C(1+s_1e^{s_1-s})}{1-e^{s_1-s}}$$ on $M\times[1,s_1]$
             for some constant $C(n, \beta, ||f||_\infty, K)>0$.
             \item[(v)] For $0<s_1<s$,
             $$u\ge -\frac{C(1+s_1e^{s_1-s})}{1-e^{s_1-s}}$$ on $M\times[0,s_1]$ for some constant $C(n, \beta, ||f||_\infty, K)>0$.
           \end{enumerate}

\end{lma}
\begin{proof} In the following, $C_i$ will denote positive constants depending only on $n, \beta, ||f||_\infty, K$ and $D_i$ will denote positive constants which may also depend on $\rho_0, \e$ but not on $\kappa$.

  To prove (i): Consider \bee
F=u(x,t)-\frac{1-e^{-t}}{1-e^{-s}}f(x)+A\cdot t-nt\log(1-e^{-t})+\kappa\rho(x) .\eee
Suppose  $\inf\limits_{M\times[0, s_1]}F<0$.
Then there exists $(x_0, t_0)\in M\times(0, s_1]$ such that $F\geq F(x_0, t_0)$ on $M\times[0, s_1]$. At this point, we have
\bee\begin{split}
0\geq &\ppt F
\\=&\dot u+A-\frac{e^{-t}}{1-e^{-s}}f(x)-n\log(1-e^{-t})-\frac{nt}{e^t-1}.\\
=&\log\frac{(-\Ric(\theta_0)+e^{-t}(\Ric(\theta_0)+\gamma_0)+\ddb u)^n}{\theta_0^n}-u +A\\
&-n\log(1-e^{-t})-\frac{nt}{e^t-1}-\frac{e^{-t}}{1-e^{-s}}f\\
\geq& \log\frac{(-\Ric(\theta_0)+e^{-t}(\Ric(\theta_0)+\gamma_0)+ \frac{1-e^{-t}}{1-e^{-s}}\ddb f-\kappa\ddb\rho)^n}{\theta_0^n}\\
&-C(n, K)-\frac{e^{-t}}{1-e^{-s}}f+A-n\log(1-e^{-t})-\frac{nt}{e^t-1},\\
\end{split} \eee
where we have used the fact that $u\leq C(n, K)$, and $\ddb u\ge\frac{1-e^{-t}}{1-e^{-s}}\ddb f-\kappa\ddb \rho$. Note that
$$
-\Ric(\theta_0)\ge\frac1{1-e^{-s}}\lf(\beta\theta_0-e^{-s}\omega_0-\ii\ddbar f\ri),
$$
hence
\bee
\begin{split}
&-\Ric(\theta_0)+e^{-t}(\Ric(\theta_0)+\gamma_0)+\frac{1-e^{-t}}{1-e^{-s}}\ddb f-\kappa\ddb\rho\\
\ge&e^{-t}\gamma_0+\frac{1-e^{-t}}{1-e^{-s}}\lf(\beta\theta_0-e^{-s}\omega_0 \ri) -\kappa K\theta_0 \\
\ge& \frac{1}{2}\frac{1-e^{-t}}{1-e^{-s}} \beta\theta_0
\end{split}
\eee
if $\kappa $ is small enough. Here we have used the fact that $0<t<s$ and $\gamma_0\ge \omega_0$.  Hence at $(x_0,t_0)$,
\bee
\begin{split}
0\geq& n\log(1-e^{-t}) -C_1 \\
&-\frac{e^{-t}}{1-e^{-s}}f+A-n\log(1-e^{-t})-\frac{nt}{e^t-1}\\
 \geq& -\frac{1}{1-e^{-s}}||f||_\infty+A-C_2 \\
 >&0
\end{split}
\eee
if $A=\frac{1}{1-e^{-s}}||f||_\infty+C_2+1$. Hence for such $A$, $F\ge 0$ and for all $\kappa>0$ small enough, we conclude that
$$
u(x,t)\ge -At+nt\log(1-e^{-t}).
$$

To prove (ii),  we have

\bee
\heat(\dot u+u)=-\tr_\omega(\Ric(\theta_0))-n.
\eee
On the other hand, by \eqref{e-udot-1}, we also have
\bee
\heat (e^t\dot u)=-\tr_\omega(\Ric(\theta_0)+\gamma_0).
\eee
Hence
\be\label{e-udot-2}
\begin{split}
& \heat\lf((1-e^{t-s})\dot u+u\ri)\\
=&\tr_\omega(-\Ric(\theta_0)+e^{-s}(\Ric(\theta_0)+\gamma_0))-n\\
\ge&\beta\tr_\omega(\theta_0)-\Delta f-n.
\end{split}
\ee

Let $F=(1-e^{t-s})\dot u+u-f-A\log t+\kappa\rho$, where $A>0$ is a constant to be determined. Since $\log t\to-\infty$ as $t\rightarrow 0$, we conclude that for $0<s_1<s$, if $\inf_{M\times[0,s_1]}F\le 0$, then there is $(x_0,t_0)\in M\times(0,s_1]$ so that
$F(x_0,t_0)=\inf_{M\times[0,s_1]}F$. By \eqref{e-udot-2}, at $(x_0,t_0)$ we have
\bee
\begin{split}
0\ge&\heat F\\
\ge&\beta\tr_\omega(\theta_0)-n-\frac At-\kappa D_1\\
\ge& n\beta\exp(-\frac1n(\dot u+u))-n-\frac At-\kappa D_1
\end{split}
\eee
where $D_1>0$ is a constant independent of $\kappa$.  Hence at this point,
$$
\dot u+u\ge -n\log\lf(\frac1{n\beta}(n+\frac At+\kappa D_1)\ri).
$$

Hence at $(x_0,t_0)$, noting that $0<t_0\le s_1<s$ and $s_1\le 1$,
\bee
\begin{split}
F\geq& (1-e^{t-s})(\dot u+u)+e^{t-s}u-f-A\log t\\
\ge&-(1-e^{t-s})n\log\lf(\frac1{n\beta}(n+\frac At+\kappa D_1)\ri)-\sup_M f-A\log t\\
&- \frac{C_3}{1-e^{-s}}+nt\log(1-e^{-t})\\
\ge&[(1-e^{t-s})n-A]\log t-(1-e^{t-s})n\log\lf(\frac1{n\beta}(nt+A+\kappa t D_1)\ri)\\
&-||f||_\infty-\frac{C_4}{1-e^{-s}} \\
\ge &- n\log\lf(\frac1{n\beta}(2n+\kappa  D_1)\ri)-||f||_\infty-\frac{C_4}{1-e^{-s}}\end{split}
\eee
if $A=n$. Here we may assume that $\beta>0$ is small enough so that $2/\beta>1$.  Hence we have
\begin{align*}
F\ge - n\log\lf(\frac1{n\beta}(2n+\kappa  D_1)\ri)-||f||_\infty-\frac{C_4}{1-e^{-s}}.
\end{align*}
on $M\times(0,s_1]$. Let $\kappa\to0$, we conclude that
\bee
\begin{split}
(1-e^{t-s})\lf(\dot u+u\ri)=& (1-e^{t-s})\dot u+u-e^{t-s}u \\
\ge&n\log t-\frac{C_5}{1-e^{-s}},
\end{split}
\eee
where we have used the upper bound of $u$ in Lemma \ref{l-uudot-upper-1}. From this (ii) follows because $t\le s_1$.

 The proof of (iii) is similar to the proof of (ii) by letting $A=0$. Note that in this case, the infimum of $F$ may be attained at $t=0$ which depends also on $\e$.

 To prove (iv), let
 $F$ as in the proof of (ii) with $A=0$. Suppose $\inf_{M\times[\frac 12,s_1]}F=\inf_{M\times\{\frac 12\} }F$, then by (i) and (ii), we have
 $$
 F\ge -C_6.
 $$
 Suppose $\inf_{M\times[\frac 12,s_1]}F<\inf_{M\times\{\frac 12\}}F$, then we can find $(x_0,t_0)\in M\times(\frac 12,s_1]$ such that $F(x_0,t_0)$ attains the infimum. As in the proof of (ii), at this point,
 \bee
 \begin{split}
 \dot u+u\ge-n\log\lf(\frac1{n\beta}(n+\kappa D_2)\ri)
 \end{split}
 \eee
where $D_2>0$ is a constant independent of $\kappa$.  Hence as in the proof of (ii),
\bee
\begin{split}
F(x_0,t_0)\ge &(1-e^{t_0-s})(\dot u+u)+e^{t_0-s}u-f\\
\geq&-n(1-e^{t_0-s}) \log\lf(\frac1{n\beta}(n+\kappa D_2)\ri)-  \frac{C_7s_1e^{s_1-s}}{1-e^{-s}}  -C_8\\
\ge&-n  \log\lf(\frac1{n\beta}(n+\kappa D_2)\ri)-  \frac{C_7s_1e^{s_1-s}}{1-e^{-s}}  -C_8
\end{split}
\eee
because $t_0\le s_1$, where we have used (i) and we may assume that $\beta<1$.
Let $\kappa\to0$, we conclude that on $M\times[\frac 12, s_1]$,
\bee\begin{split}
&(1-e^{t-s})(\dot u+u)+e^{t-s}u-f\ge n  \log\beta-  \frac{C_7s_1e^{s_1-s}}{1-e^{-s}}  -C_8.\end{split} \eee

By Lemma \ref{l-uudot-upper-1}, we have
\bee
\dot u+u\ge -\frac{C_9(1+s_1e^{s_1-s})}{1-e^{s_1-s}}
\eee on $M\times[\frac 12,s_1]$ for some constant because $s>1$.

Finally, (v) follows from (i), Lemma \ref{l-uudot-upper-1} and (iv) by integration.

\end{proof}

\subsection{a priori estimates for $\omega(t)$}\label{ss-trace}

Next we will estimate the uniform upper bound of $g(t)$. Before we do this, we first give uniform estimates for the evolution of the key quantity $\log \tr_hg(t)$.

Let $\hat T$ and $T_0$ be the torsions of $h, \gamma_0$ respectively. Note that $\gamma_0$ depends on $\rho_0, \e$. Let $\hat\nabla$ be the Chern connection of $h$. Recall that $T_{ij\bar l}=\p_ig_{j\bar l}-\p_j g_{i\bar l}$ etc.

Let $\wt g$ be such that $g(t)=e^{-t}\wt g(e^t-1)$. Let  $s=e^t-1$. Then
\bee
\begin{split}
-\Ric (\wt g(s))-g(t)=&-\Ric (g(t))-g(t)\\
=&\frac{\p}{\p t}g(t)\\
=&-e^{-t}\wt g(e^t-1)+\frac{\p}{\p s}\wt g(s)\\
=&-g(t)+\frac{\p }{\p s}\wt g(s).
\end{split}
\eee
So
\bee
\frac{\p }{\p s}\wt g(s)=-\Ric(\wt g(s))
\eee
and $\wt g(0)=\gamma_0$.

Let $\Upsilon(t)=\tr_{h}g(t)$ and $\wt\Upsilon(s)=\tr_{h}\wt g(s)$.
By Lemma \ref{l-a-1}, we have
\bee
  \sheat \log \wt\Upsilon=\mathrm{I+II+III}
  \eee
  where
  \bee
\begin{split}
\mathrm{I}\le &2\wt\Upsilon^{-2}\text{\bf Re}\lf( h^{i\bar l} \wt g^{k\bar q} (T_0)_{ki\bar l}\hat \nabla_{\bar q}\wt\Upsilon\ri).
\end{split}
\eee
\bee
\begin{split}
\mathrm{II}=&\wt\Upsilon^{-1} \wt g^\ijb  h^{k\bar l}\wt g_{k\bar q} \lf(\hat \nabla_i \ol{(\hat T)_{jl}^p}-   h^{p\bar q}\hat R_{i\bar lp\bar j}\ri)\\
\end{split}
\eee
and
  \bee
\begin{split}
\mathrm{III}=&-\wt\Upsilon^{-1} \wt g^{\ijb}  h^{k\bar l}\lf(\hat \nabla_i\lf(\ol{( T_0)_{jl\bar k} } \ri) +\hat \nabla_{\bar l}\lf(  (  T_0)_{ik\bar j}  \ri)-\ol{ (\hat T)_{jl}^q}(  T_0)_{ik\bar q}   \ri)
\end{split}
\eee

Now
$$
\wt \Upsilon(s)=e^t\Upsilon(t).
$$
So
\bee
\sheat \log\wt\Upsilon(s)=e^{-t}\lf(\heat  \log\Upsilon+1\ri)
\eee
\bee
\begin{split}
\mathrm{I}\le &2e^{- 2t}\Upsilon^{-2}\text{\bf Re}\lf( h^{i\bar l}   g^{k\bar q} (T_0)_{ki\bar l}\hat \nabla_{\bar q} \Upsilon\ri).
\end{split}
\eee
\bee
 \begin{split}
\mathrm{II}=&e^{-t}\Upsilon^{-1}   g^\ijb  h^{k\bar l}  g_{k\bar q} \lf(\hat \nabla_i \ol{(\hat T)_{jl}^q}-   h^{p\bar q}\hat R_{i\bar lp\bar j}\ri)\\
\end{split}
\eee
and
  \bee
\begin{split}
\mathrm{III}=&-e^{-2t}\Upsilon^{-1}  g^{\ijb}  h^{k\bar l}\lf(\hat \nabla_i\lf(\ol{( T_0)_{jl\bar k} } \ri) +\hat \nabla_{\bar l}\lf(  (  T_0)_{ik\bar j}  \ri)-\ol{ (\hat T)_{jl}^q}(  T_0)_{ik\bar q}   \ri)
\end{split}
\eee

Hence
\be\label{e-logY}
\heat\log \Upsilon=\mathrm{I}'+\mathrm{II}'+\mathrm{III}'-1
\ee
where
\bee
\begin{split}
\mathrm{I}'\le &2e^{-t}\Upsilon^{-2}\text{\bf Re}\lf( h^{i\bar l}   g^{k\bar q} (T_0)_{ki\bar l}\hat \nabla_{\bar q} \Upsilon\ri).
\end{split}
\eee
\bee
\begin{split}
\mathrm{II}'=& \Upsilon^{-1}   g^\ijb  h^{k\bar l}  g_{k\bar q} \lf(\hat \nabla_i \ol{(\hat T)_{jl}^q}-   h^{p\bar q}\hat R_{i\bar lp\bar j}\ri)\\
\end{split}
\eee
and
  \bee
\begin{split}
\mathrm{III}'=&-e^{-t}\Upsilon^{-1}  g^{\ijb}  h^{k\bar l}\lf(\hat \nabla_i\lf(\ol{( T_0)_{jl\bar k} } \ri) +\hat \nabla_{\bar l}\lf(  (  T_0)_{ik\bar j}  \ri)-\ol{ (\hat T)_{jl}^q}(  T_0)_{ik\bar q}   \ri)
\end{split}
\eee

Now we want to estimate the terms in the above differential inequality.

\underline{\it Estimate of $\mathrm{II}'$}

Choose an frame unitary with respect to $h$ so that $g_{\ijb}=\lambda_i\delta_{ij}$. Then
\be\label{e-logY-1}
\begin{split}
\mathrm{II}'=& (\sum_l\lambda_l)^{-1}\lambda_i^{-1}\lambda_k\lf(\hat\nabla_i\ol{(\hat T)_{ik}^k}-\hat R_{i\bar kk\bar i}\ri)\\
\le &C(n,K)\tr_{g}h.
\end{split}
\ee

 \underline{\it Estimate of $\mathrm{III}'$}

Next, we compute the torsion of $\gamma_0$, $T_0=T_{\gamma_0}$, where $\gamma_0=\eta(\frac{\rho(x)}{\rho_0})g_0+(1-\eta(\frac{\rho(x)}{\rho_0}))h+\e h$:\bee\begin{split}
(T_0)_{ik\bar q}=&\p_i(\gamma_0)_{k\bar q}-\p_k(\gamma_0)_{i\bar q}\\
=&\eta'\frac{1}{\rho_0}[\rho_i(x)(g_0)_{k\bar q}-\rho_k(x)(g_0)_{i\bar q}]+\eta[\p_i(g_0)_{k\bar q}-\p_k(g_0)_{i\bar q}]\\
&+(1-\eta+\e)[\p_ih_{k\bar q}-\p_kh_{i\bar q}]-\eta'\frac{1}{\rho_0}[\rho_ih_{k\bar q}-\rho_kh_{i\bar q}]. \end{split} \eee

By the assumptions, all terms above are bounded by $C(n, K)$ for all $\rho_0\geq 1$ and for all $\e\leq 1$.

It remains to control $\hat \nabla_{\bar l}\lf(  (  T(\gamma_0))_{ik\bar j}  \ri)$.  We may compute $\hat \nabla_{\bar l}\lf(  (  T(\gamma_0))_{ik\bar j}  \ri)$ directly. \bee\begin{split}
& \hat \nabla_{\bar l}\lf(  (  T(\gamma_0))_{ik\bar j}  \ri)\\=&\hat \nabla_{\bar l}(\p_i(\gamma_0)_{k\bar j}-\p_k(\gamma_0)_{i\bar j})\\
=&\hat \nabla_{\bar l}\{\eta'\frac{1}{\rho_0}[\rho_i(x)(g_0)_{k\bar j}-\rho_k(x)(g_0)_{i\bar j}]+\eta[\p_i(g_0)_{k\bar j}-\p_k(g_0)_{i\bar j}]\\
&+(1-\eta+\e)[\p_ih_{k\bar j}-\p_kh_{i\bar j}]-\eta'\frac{1}{\rho_0}[\rho_ih_{k\bar q}-\rho_kh_{i\bar q}]\}\\
=&\eta''\rho_{\bar l}\frac{1}{\rho^2_0}[\rho_i(g_0)_{k\bar j}-\rho_k(g_0)_{i\bar j}]+\eta'\frac{1}{\rho_0}[\rho_{i\bar l}(g_0)_{k\bar j}-\rho_{k\bar l}(g_0)_{i\bar j}]\\
&+\eta'\frac{1}{\rho_0}[\rho_i\hat \nabla_{\bar l}(g_0)_{k\bar j}-\rho_k\hat \nabla_{\bar l}(g_0)_{i\bar j}]+\eta_{\bar l}[\p_i(g_0)_{k\bar j}-\p_k(g_0)_{i\bar j}]\\
&+\eta \hat\nabla_{\bar l} T(g_0)_{ik\bar q}+(1-\eta+\e)\hat \nabla_{\bar l} T(h)_{ik\bar j}-\eta'\frac{\rho_{\bar l}}{\rho_0}T(h)_{ik\bar j}\\
&-\eta'\frac{1}{\rho_0}[\rho_{i\bar l}h_{k\bar q}-\rho_{k\bar l}h_{i\bar q}]-\eta''\frac{1}{\rho^2_0}[\rho_{\bar l}\rho_{i}h_{k\bar q}-\rho_{\bar l}\rho_{k}h_{i\bar q}]. \end{split} \eee

Since we can control every term of the above equation by $C(n, K)$. Therefore, $|\hat \nabla_{\bar l}\lf(  (  T(\gamma_0))_{ik\bar j}  \ri)|\leq C(n, K)$.

Therefore, if $0<\e<1,\rho_0>1$
\be\label{e-logY-2}
\mathrm{III}'\leq C(n, K)\cdot e^{-t}\Upsilon^{-1} \Lambda.
 \ee
 where $\Lambda=\tr_{g}h$.

Now we will prove the uniform upper bound of $g(t)$.

 \begin{lma}\label{l-trace-2}
\begin{enumerate}
 \item [(i)] For $0<s_1<s$,
  $$
 \tr_{h}g(x,t)\le \exp\lf(\frac{C(E-\log(1-e^{-s}))}{1-e^{-t}}\ri)
  $$
  on $M\times(0,s_1]$ for some   constant $C>0$ depending only on $n,K, \beta, ||f||_\infty$ provided  such that if $0<\e<1$, $\rho_0>1$,
  where
  $$
  E=\frac{(1+s_1e^{s_1-s})}{(1-e^{-s})(1-e^{s_1-s})}.
  $$
  \item [(ii)] For $0<s_1<s$, there is  a constant $C$  depending only on $n,K, \beta, ||f||_\infty,   s, s_1$ and also on $\e$,  but independent of $\rho_0$ such that
$$
\tr_{h}g\le C
$$
on $M\times[0, s_1]$.
\end{enumerate}

\end{lma}
\begin{proof} In the following, $C_i$ will denote constants depending only on $n,K, \b$ and $||f||_\infty$, but not $\rho_0$ and $\e$. $D_i$ will denote constants which may also depend on $\e, \rho_0$, but not $\kappa$. We always assume $0<\e<1<\rho_0$.

 Let $v(x,t)\ge1$ be a smooth bounded function.
As before, let  $\Upsilon=\tr_{h}g$ and $\Lambda=\tr_gh$ and let $\lambda=0$ or 1. For $\kappa>0$, consider the function
$$F=(1-\lambda e^{-t})\log \Upsilon-Av+\frac 1v-\kappa\rho+Bt\log (1-\lambda e^{-t})
 $$
 on $M\times[0, s_1]$, where $A, B>0$ are constants  to be chosen.  We want to estimate $F$ from above.
 Let
 $$
  \mathfrak{M}=\sup_{M\times[0,s_1]}F.
  $$
  Either (i) $\mathfrak{M}\le 0$; (ii) $\mathfrak{M}=\sup_{M\times\{0\}}F$; or (iii) there is $(x_0,t_0)$ with $t_0>0$ such that $F(x_0,t_0)=\mathfrak{M}$.
If (ii) is true, then
 \be\label{e-tr-1}
 \mathfrak{M}\le C_1(n).
 \ee
 because $g(0)=\gamma_0\le (1+\e)h$.

 Suppose (iii)  is true. If at this point $\Upsilon(x_0,t_0)\le 1$. Then \eqref{e-tr-1} is true with a possibly larger $C_1$. So let us assume that $\Upsilon(x_0,t_0)>1$. By \eqref{e-logY}, \eqref{e-logY-1} and \eqref{e-logY-2}, at $(x_0,t_0)$ we have:
\bee
\begin{split}
0\le&\heat F\\
=&(1-\lambda e^{-t})\heat\log \Upsilon+\lambda e^{-t}\log \Upsilon-(\frac{1}{v^2}+A)\heat v\\
&-\frac{2}{v^3}|\n v|^2+\kappa \Delta \rho+B\lf(\log(1-\lambda e^{-t})+\frac{\lambda t}{e^t-\lambda}\ri)\\
\le &(1-\lambda e^{-t})C_2\Lambda \lf( 1 +e^{-t}\Upsilon^{-1} \ri)\\& +
2(1-\lambda e^{-t})e^{-t}\Upsilon^{-2}\text{\bf Re}\lf( h^{i\bar l}   g^{k\bar q} (T_0)_{ki\bar l}\hat \nabla_{\bar q} \Upsilon\ri)\\
&+\lambda e^{-t}\log \Upsilon- (\frac 1{v^2}+A)\heat v -\frac{2|\nabla v|^2}{v^3}\\
&+B\lf(\log(1-\lambda e^{-t})+\frac{\lambda t}{e^t-\lambda}\ri)+\kappa D_1.
\end{split}
\eee
 At $(x_0,t_0)$, we also have:
  $$(1-\lambda e^{-t}) \Upsilon^{-1}\hat \nabla \Upsilon-(\frac 1{v^2}+A)\hat\nabla v- \kappa\hat  \nabla \rho=0.$$

Hence
 \bee\begin{split}
 &2(1-\lambda e^{-t})e^{-t} \Upsilon^{-2}\text{\bf Re}\lf( h^{i\bar l}   g^{k\bar q} (T_0)_{ki\bar l}\hat \nabla_{\bar q} \Upsilon\ri)\\
=& \frac{2e^{-t}}{\Upsilon}\text{\bf Re}\lf( h^{i\bar l}   g^{k\bar q} (T_0)_{ki\bar l}((\frac 1{v^2}+A)\hat\nabla_{\bar q} v- \kappa\hat  \nabla_{\bar q} \rho)\ri) \\
\leq&\frac{1}{v^3}|\nabla v|^2+\frac{C_3(A+1+\frac{1}{v^2})^2\cdot v^3 \Lambda}{\Upsilon^2}+\kappa D_2\\
\end{split}\eee

Using the fact that $\Upsilon(x_0,t_0)>1$, we have at $(x_0,t_0)$:

Hence \be\label{e-g-1}\begin{split}
0\le& C_2(1-\lambda e^{-t})\Lambda  +
\frac{C_3(A+\frac{1}{v^2})^2\cdot v^3 \Lambda}{\Upsilon} +\lambda e^{-t}\log \Upsilon\\
&- (\frac 1{v^2}+A)\heat v +B\lf(\log(1-\lambda e^{-t})+\frac{\lambda t}{e^t-\lambda}\ri)+\kappa D_3.\end{split}
\ee
Now let
$$
 v=u-\frac{1- e^{-t}}{1-e^{-s}}f+\frac{C_4(1+s_1e^{s_1-s})}{(1-e^{-s})(1-e^{s_1-s})}
$$
 By Lemmas \ref{l-uudot-upper-1} and  \ref{l-all-u}, we can find $C_4>0$ so that $v\ge 1$, and there is $C_5>0$ so that $$v\le  \frac{C_5(1+s_1e^{s_1-s})}{(1-e^{-s})(1-e^{s_1-s})}.
 $$
Let
\be\label{e-E}
E:=\frac{(1+s_1e^{s_1-s})}{(1-e^{-s})(1-e^{s_1-s})}.
\ee

Note that
 \bee
\begin{split}
& \heat u\\
=&\dot u-\Delta u\\
=&\dot u-n+\tr_g\lf(-(1-e^{-t})  \Ric(\theta_0)+e^{-t}\gamma_0 \ri)\\
\ge&\dot u-n+\tr_g\lf(\frac{1-e^{-t}}{1-e^{-s}}\lf(\beta\theta_0-e^{-s}\omega_0-\ii\ddbar f\ri)+ e^{-t}\gamma_0\ri)\\
\ge&\dot u+\lf[\frac{\beta(1-e^{-t})}{1-e^{-s}}+\e e^{-t}\ri]\Lambda-\frac{1-e^{-t}}{1-e^{-s}}\Delta f -n\\
\ge&\dot u+\lf[\frac{\beta(1-e^{-t})}{1-e^{-s}}+\e e^{-t}\ri]\Lambda+\heat \lf(\frac{1- e^{-t}}{1-e^{-s}}  f\ri)-\frac{ e^{-t}}{1-e^{-s}}f-n\\
\ge&\dot u+u+\lf[\frac{\beta(1-e^{-t})}{1-e^{-s}}+\e e^{-t}\ri]\Lambda+\heat \lf(\frac{1- e^{-t}}{1-e^{-s}}  f\ri)-\frac{C_6}{1-e^{-s}}.
\end{split}
\eee
because $\gamma_0\ge \omega_0+\e\theta_0$ and $t<s$.

On the other hand,
$$
-\dot u-u=\log\lf(\frac{\det h}{\det g}\ri)\le c(n)+n\log \Lambda.
$$
Hence
\be\label{e-g-2}
\heat v\ge -n\log \Lambda+ \lf[\frac{\beta(1-e^{-t})}{1-e^{-s}}+\e e^{-t}\ri]\Lambda -\frac{C_7}{1-e^{-s}}.
\ee
On the other hand, in a unitary frame with respect to $h$ so that $g_\ijb=\lambda_i\delta_{ij}$, then
 \be\label{e-tr-2}
 \begin{split}
 \Upsilon=&\sum_i\lambda_i\\
 =&\frac{\det g}{\det h}\sum_{i} (\lambda_1\dots\hat\lambda_i\dots\lambda_n)^{-1}\\
 \le &C_{8}\Lambda^{n-1}.
 \end{split}
 \ee
where we have used the upper bound of $\dot u+u=\log\frac{\det g}{\det h} $ in Lemma \ref{l-uudot-upper-1}. Combining \eqref{e-g-1}, \eqref{e-g-2} and \eqref{e-tr-2},  at $(x_0,t_0)$ we have
\be \label{e-tr-revised}
\begin{split}
0\le& C_2(1-\lambda e^{-t})\Lambda\lf(1  +
\frac{C_9  E^3(A+1)^2}{(1-\lambda e^{-t})\Upsilon}\ri)  +\lambda e^{-t}\lf(\log C_8+(n-1)\log\Lambda\ri)\\
&+ (\frac 1{v^2}+A)\lf( n\log \Lambda- \lf[\frac{\beta(1-e^{-t})}{1-e^{-s}}+\e e^{-t}\ri]\Lambda+\frac{C_7}{1-e^{-s}}\ri) \\
 &+B\lf(\log(1-\lambda e^{-t})+\frac{\lambda t}{e^t-\lambda}\ri)+\kappa D_3\\
 \le&\Lambda\left[C_2(1-\lambda e^{-t}) \lf(1  +
\frac{C_9E^3(A+1)^2}{(1-\lambda e^{-t})\Upsilon}\ri)-\frac{A+1}{C^2_5E^2}\lf(\frac{\beta(1-e^{-t})}{1-e^{-s}}+\e e^{-t}\ri)\ri] \\
&+[n(1+A)+\lambda(n-1)] \log \Lambda+ \frac{C_{10}(A+1) }{1-e^{-s}}\\&+B\lf(\log(1-\lambda e^{-t})+\frac{\lambda t}{e^t-\lambda}\ri)+\kappa D_3+\lambda\log C_8 \end{split}
\ee
where we have used the fact that $1\le v\le C_5E$.

{\bf Case 1}:  Let  $\lambda=1$. Suppose at $(x_0,t_0)$,
$$
\frac{C_2C_9E^3(A+1)^2}{(1-  e^{-t})\Upsilon}\ge \frac12\frac{1}{C^2_5E^2}\cdot(A+1)\cdot\beta\cdot\frac{1}{1-e^{-s}}
$$
Then
 \bee (1-  e^{-t})\Upsilon\le \frac{2C_2C_9C^2_5E^5(1-e^{-s})(A+1)}{\beta}\leq C_{11}E^5(A+1).
 \eee
Hence,
\bee
(1-  e^{-t})\log \Upsilon\le (1-  e^{-t})\log(C_{11}E^5(A+1)) -(1-e^{-t})\log(1-  e^{-t}) .
\eee
Therefore,
\be\label{e-tr-1-2}
\mathfrak{M}\le C(1+\log E)+\log (A+1).
\ee
for some $C(n,\beta, K,||f||_\infty)>0$. Suppose at $(x_0,t_0)$,
 $$
\frac{C_2C_9E^3(A+1)^2}{(1-  e^{-t})\Upsilon}< \frac12\frac{1}{C^2_5E^2}\cdot(A+1)\cdot\beta\cdot\frac{1}{1-e^{-s}},
$$  then at $(x_0,t_0)$ we have

\bee
\begin{split}
0\le &  (1-  e^{-t}) \Lambda  \lf(C_2  -
\frac12\frac{1}{C^2_5E^2}\cdot(A+1)\cdot\beta\cdot\frac{1}{1-e^{-s}}\ri) +n(A+2)\log \Lambda\\
&+ \frac{C_{10}(A+1) }{1-e^{-s}} +B\lf(\log(1-  e^{-t})+\frac{  t}{e^t-1}\ri)+\kappa D_3+\log C_8\\
=&\Lambda\lf[ (1-  e^{-t}) \lf(C_2  -
\frac12\frac{1}{C^2_5E^2}\cdot(A+1)\cdot\beta\cdot\frac{1}{1-e^{-s}} \ri)\ri]\\
&+n(A+2)\log ((1-e^{-t})\Lambda)+\frac{C_{10}(A+1) }{1-e^{-s}}-n(A+2)\log(1-e^{-t})\\
&+B\lf(\log(1-  e^{-t})+\frac{  t}{e^t-1}\ri)+\kappa D_3+\log C_8\\
\le &-(1-e^{-t})\Lambda +n(A+2)\log ((1-e^{-t})\Lambda)+\frac{C_{12}E^2}{1-e^{-s}},
\end{split}
\eee
provided $A=C_{13}E^2$ so that
$$
\frac12\frac{1}{C^2_5E^2}\cdot(A+1)\cdot\beta\cdot\frac{1}{1-e^{-s}} \ge (C_2+1)
$$
and
   $B$ is chosen so that $B=n(A+2)$ and $\kappa$ is small enough so that $\kappa D_2\le 1$.
 Hence using $1+\frac{1}{2}\log x\leq \sqrt{x},\;\forall x>0$, we have at $(x_0,t_0)$,
$$
(1-e^{-t})\Lambda\le \frac{C_{14}E^4}{1-e^{-s}},
$$
and so
$$
\log\Lambda\le \log\frac{C_{14}E^4}{1-e^{-s}}-\log(1-e^{-t}).
$$
By \eqref{e-tr-2}, we have

 \be\label{e-tr-4}
 \begin{split}
&(1-e^{-t})\log\Upsilon\\
\leq&(1-e^{-t})\lf( \log C_8+(n-1)\log \Lambda\ri) \\
\le&(1-e^{-t})\lf( \log C_8+(n-1)\lf(\log\frac{C_{14}E^4}{1-e^{-s}}-\log(1-e^{-t})\ri)\ri)\\
\le &(n-1)\log(\frac{1}{1-e^{-s}})+C_{15}(1+\log E).
\end{split}
\ee
Hence $\mathfrak{M}\le (n-1)\log(\frac{1}{1-e^{-s}})+C_{16}(1+\log E).$
By combining \eqref{e-tr-1}, \eqref{e-tr-1-2}  and using the choice of $A$, we may let $\kappa\rightarrow 0$ to conclude that on $ M\times(0,s_1]$,
$$
(1-e^{-t})\log \Upsilon\le (n-1)\log(\frac{1}{1-e^{-s}})+C_{17}(1+E).
$$
and hence (i) in the lemma is true. Here we have used the fact that $E\geq \log E+1$.

{\bf Case 2}: Let $\lambda=0$, then \eqref{e-tr-revised} becomes:
\bee
\begin{split}
0\le&\Lambda\lf[C_2\lf(1  +
\frac{C_9E^3(A+1)^2}{\Upsilon}\ri)-\frac{1}{C^2_5E^2}(A+1)\e e^{-t}\ri] \\
&+n(1+A) \log \Lambda+ \frac{C_{10}(A+1) }{1-e^{-s}}+\kappa D_3. \end{split}
\eee
We can argue as before to conclude that (ii) is true.

\end{proof}

Combining the lower bound of $\dot u+u$, we obtain:

\begin{cor}\label{eq-g} For any $0<s_0<s_1<s$, there is a constant $C$ depending only on $n,K, \b, ||f||_\infty$ and $s_0, s_1, s$  but independent of $\e,\rho_0$ such that if $0<\e<1$, $\rho_0>1$, we have
\bee
 C^{-1}h\leq g(t)\leq Ch \eee on $M\times[s_0, s_1]$. There is also a constant $\wt C(\e)>0$ which may also depend on $\e$ such that

 \bee
 \wt C^{-1}h\leq g(t)\leq \wt Ch
 \eee
 on $M\times[0, s_1]$.

\end{cor}

\section{Short time existence for the potential flow and the normalized Chern-Ricci flow}

Using the    a priori estimates in previous section, we are ready to discuss short time existence for the   the potential flow and the Chern-Ricci flow. We begin with the short time existence of the potential flow. We have the following:

\begin{thm}\label{t-instant-complete}
Let $(M,h)$ be a complete non-compact Hermitian metric { with \K form $\theta_0$.} Suppose there is $K>0$ such that the following holds.
\begin{enumerate}
\item There is a {proper} exhaustion function $\rho(x)$ on $M$ such that
$$|\partial\rho|^2_h +|\ddb \rho|_h \leq K.$$
\item $\mathrm{BK}_h\geq -K$;
\item The torsion of $h$, $T_h=\partial \omega_h$ satisfies
$$|T_h|^2_h +|\nabla^h_{\bar\partial} T_h |\leq K.$$
\end{enumerate}
Let $\omega_0$ be a nonnegative real (1,1) form with corresponding Hermitian form  $g_0$   on $M$ (possibly incomplete or degenerate) such that
\begin{enumerate}
\item[(a)] $g_0\le h$ and
$$|T_{g_0}|_h^2+|\nabla^h_{\bar\partial} T_{g_0}|_h+ |\nabla^{h}g_0|_h\leq K.$$

\item[(b)] There exist $f\in C^\infty(M)\cap L^\infty(M),\beta>0$ and $s>0$ so that  $$-\Ric(\theta_0)+e^{-s}(\omega_0+\Ric(\theta_0))+\ddb f\geq \beta \theta_0.$$

\end{enumerate}
Then \eqref{e-MP-1} has a solution on $M\times(0, s)$ so that $u(t)\to 0$ as $t\to0$ uniformly on $M$. Moreover, for any $0<s_0<s_1<s$, { let
$$
\a(t)=-\Ric(\theta_0)+e^{-t}(\Ric(\theta_0)+\omega_0)
$$}
then
$$\omega(t)=\a+\ii\ddbar u$$
is the \K form of a complete Hermitian metric which is uniformly equivalent to $h$ on $M\times[s_0, s_1]$.
\end{thm}

\begin{proof}[Proof of Theorem \ref{t-instant-complete}] For later application, we construct the solution in the following way. Combining the local higher order estimate of Chern-Ricci flow (See \cite{ShermanWeinkove2013} for example) with Corollary \ref{eq-g} for any $1>\e>0$,  using diagonal argument as $\rho_0\to \infty$ we obtain a solution $u_\e(t)$ to    \eqref{e-MP-1} with initial data $\omega_0+\e \theta_0$ on $M\times[0,s)$ which is smooth up to $t=0$, so that the corresponding solution $g_\e(t)$ of \eqref{e-NKRF} has smooth solution on $M\times[0,s)$ with initial metric $g_\e(0)=g_0+\e h$. Moreover, $g_\e$ is uniformly equivalent to $h$ on $M\times[0,s_1]$ for all $0<s_1<s$ and  for any $0<s_0<s_1<s$, there is a constant $C>0$ independent of $\e$ such that
$$
C^{-1}h\le g_\e\le Ch
$$
on $M\times[s_0,s_1]$.
Using the local higher order estimate of Chern-Ricci flow   \cite{ShermanWeinkove2013} again, we can find $\e_i\to0$ such that $u_{\e_i}$ converge locally uniformly on any compact subsets of $M\times(0,s)$ to a solution $u$ of \eqref{e-MP-1}.
 By Lemmas \ref{l-uudot-upper-1}, \ref{l-all-u},   we see that $u(t)\to 0$ as $t\to0$ uniformly $M$. Moreover, for any $0<s_0<s_1<s$, $\omega(t)=\a+\ii\ddbar u$ is the \K form of the solution to \eqref{e-NKRF}. Also, the corresponding Hermitian metric $g(t)$ is a complete Hermitian metric which is uniformly equivalent to $h$ in $M\times[s_0, s_1]$ for any $0<s_0<s_1<1$.

\end{proof}

Next we want to discuss { the short time existence of the Chern-Ricci flow. The solution $\omega(t)$ obtained from the Theorem \ref{t-instant-complete} satisfies the normalized Chern-Ricci flow on $M\times(0,s)$. Hence we concentrate on the discussion of the   behaviour of $\omega(t)$ as $t\to0$ for the solution obtained      in Theorem \ref{t-instant-complete}}.  In case that
 $h$ is \K and $\omega_0$ is closed, we have the following:

\begin{thm}\label{t-initial-Kahler-1}
With the same notation  and assumptions as in Theorem \ref{t-instant-complete}. Let $\omega(t)$ be the solution of \eqref{e-NKRF} obtained in the theorem.  If in addition $h$ is \K and $d\omega_0=0$. Let $U=\{\omega_0>0\}$.  Then $\omega(t)\rightarrow \omega_0$ in $C^\infty(U)$ as $t\rightarrow 0$, {uniformly in compact sets of $U$}. 
\end{thm}
\begin{rem}
If in addition $h$ has bounded curvature, then one can use Shi's \KR flow \cite{Shi1989,Shi1997} and the argument in \cite{ShermanWeinkove2012} to show that the \KR flow $g_i(t)$ starting from $g_0+\e_i h$ has bounded curvature when $t>0$. The uniform local $C^k$ estimates will follow from the pseudo-locality theorem \cite[Corollary 3.1]{HeLee2018} and the modified Shi's local estimate \cite[Theorem 14.16]{Chow2}.
\end{rem}

By Theorem \ref{t-instant-complete} we have the following:
 \begin{cor}\label{c-shorttime} Let $(M,h)$ be a complete non-compact \K manifold with bounded curvature. Let $\theta_0$ be the \K form of $h$. Suppose there is a compact set $V$ such that outside $V$, $-\Ric(\theta_0)+\ii\ddbar f\ge\beta \theta_0$ for some $\beta>0$ for some bounded smooth function $f$. Then for any closed nonnegative real (1,1) form $\omega_0$ such that $\omega_0\le \theta_0$, $|\nabla_h\omega_0|$ is bounded, and  $\omega_0>0$ on $V$, there is $s>0$ such that \eqref{e-NKRF} has a solution $\omega(t)$ on $M\times(0,s)$ so that $\omega(t)$ is uniformly equivalent to $h$ on $M\times[s_0,s_1]$ for any $0<s_0<s_1<s$ and $\omega(t)$ attains initial data $\omega_0$ in the set where $\omega_0>0$.

\end{cor}
\begin{proof} Let $s>0$, then
$$
-(1-e^{-s})\Ric(\theta_0)+(1-e^{-s})\ii\ddbar f\ge (1-e^{-s})\beta \theta_0
$$
outside $V$. On $V$,
$$
\omega_0 -(1-e^{-s})\Ric(\theta_0)+(1-e^{-s})\ii\ddbar f\ge \beta'\theta_0
$$
for some $\beta'>0$, provided $s$ is small enough. The Corollary then follows from Theorems \ref{t-instant-complete} and \ref{t-initial-Kahler-1}.
\end{proof}

\begin{rem}\label{r-shorttime}
Suppose $\Omega$ is a bounded strictly pseudoconvex domain in $\C^n$ with smooth boundary, then there is a complete \K metric with Ricci curvature  bounded above by the negative constant near infinity by \cite{ChengYau1982}. Hence Corollary \ref{c-shorttime} can be applied to this case, which has been studied by Ge-Lin-Shen \cite{Ge-Lin-Shen}.
\end{rem}

To prove the Theorem \ref{t-initial-Kahler-1},
suppose $h$ is \K and $d\omega_0=0$, then   solution in  Theorem \ref{t-instant-complete} is the limit of solutions $g_i(t)$ of the normalized  \KR flow on $M\times[0, s)$ with initial data $g_0+\e_i h$, where $\e_i\to 0$. Here we may assume $s\leq 1$.  By Lemma \ref{l-all-u} (iii) and Lemma \ref{l-trace-2} (ii),  each $g_i(t)$ is uniformly equivalent to $h$, the uniform constant here may depend on $\e_i$. In this section, we will use $\tilde g_i(t)=(t+1)g_i( \log (t+1))$ to denote the unnormalized \KR flow and $\phi_i$ be the corresponding potential flow to the unnormalized \KR flow $\tilde g_i(t)$, see appendix.

We want to prove the following:

\begin{lma}\label{l-initial-Kahler-1}   With the same notation  and assumptions as in Theorem \ref{t-initial-Kahler-1}, for any precompact open subset $\Omega$ of $U$, there is $S_1>0$ and $C>0$,
 $$
 C^{-1}h\le  \tilde g_i(t)\le Ch
 $$
 for all $i$  in $\Omega\times[0,S_1]$.
\end{lma}

{\begin{proof}[Proof of Theorem \ref{t-initial-Kahler-1}]
Suppose the lemma is true, then Theorem \ref{t-initial-Kahler-1} will  follow from the local estimates in \cite{ShermanWeinkove2012}.
\end{proof}

It remains to prove Lemma \ref{l-initial-Kahler-1}.}

\begin{lma}\label{slma-1} We have $|\phi_i|\leq C_0,\;\dot\phi_i\leq C_0$ on $M\times[0, e^s-1)$ for some positive constant $C_0$ independent of $i$.
\begin{proof} By Lemma \ref{l-uudot-upper-1}, we have \bee
\log\frac{{\omega_i}^n(s)}{\theta_0^n}=\dot u_i+u_i\leq C. \eee Here $C$ is a positive constant independent of $i$ and ${\wt\omega_i}(s)$ is the corresponding normalized flow with the relation \bee
\wt g_i(t)e^{-s}= g_i(s), t=e^s-1. \eee
Then by the equation $\dot\phi_i=\log\frac{\wt\omega^n_i(t)}{\theta_0^n}$, we obtain the upper bound on $\dot\phi_i(t)$. The lower bound on $\phi_i$ follows from Lemma \ref{l-all-u}.
\end{proof}
\end{lma}
Before we state the next lemma, we fix some notations. Let $p\in U$. By scaling, we may assume that there is a holomorphic coordinate neighbourhood of $p$ which can be identified with $B_0(2)\subset \C^n$ with $p$ being the origin and $B_0(r)$ is the Euclidean ball with radius $r$. Moreover, $B_0(2)\Subset U$. We may further assume $\frac14h\le h_E\le 4h$ in $B_0(2)$ where $h_E$ is the Euclidean metric. Since $\omega_0>0$, there is $\sigma>0$ such that $B_{g_i(0)}(p,2\sigma)\subset B_0(2)$ and
$$
g_i(0)\ge 4\sigma^2h
$$
in $B_0(2)$ for some $0<\sigma<1$. This is because $g_i(0)=\omega_0+\e_i h$.  Here we use $h_E$ because we want to use the estimates in \cite{ShermanWeinkove2012} explicitly. Let $\tau=e^{s}-1$,  where $s$ is the constant in assumption in Theorem \ref{t-instant-complete}, { and let $\dot \phi$ be as in the proof of Lemma \ref{slma-1}. It is easy to see that Lemma \ref{l-initial-Kahler-1} follows from the following:}

 \begin{lma}\label{local-bound} With the same notation and assumptions as in Theorem \ref{t-initial-Kahler-1} and with the above set up. There exist positive constants $1>\gamma_1, \gamma_2>0$ with $\gamma_2<\tau$  which are independent of $i$ such that
{ $$\gamma_1^{-2}h\ge \wt g_i(t)\geq \gamma_1^2 h$$}
on  $B_{\wt g_i(t)}(p,\sigma),\;  t\in [0,\gamma_2\gamma_1^{8(n-1)}]$.
\end{lma}
\begin{proof} The lower bound in lemma will follow from the following:\vskip .1cm

\noindent\underline{\it Claim}: There are constants $1>\gamma_1,  \gamma_2>0$ independent of $\a>0$ and $i$ with $\gamma_2<\tau$ such that if
$\wt g_i(t)\ge \a^2h$ on  $B_{\wt g_i(t)}(p,\sigma)$, $t\in [0,  \gamma_2\a^{8(n-1)}]$, then
$\wt g_i(t)\ge \gamma_1^2 h$ on $B_{\wt g_i(t)}(p,\sigma)$ for $t\in [0, \gamma_2\a^{8(n-1)}]$.
\vskip .1cm

The main point is that  $\gamma_1$ does not depend on $\a$. Suppose the claim is true. Fix $i$, let $\a\le \gamma_1$ be the supremum of $\wt\a$ so that $\wt g_i(t)\ge \wt\a^2h$ on $  B_{\wt g_i(t)}(p,\sigma)$, $t\in [0, \gamma_2\wt\a^{8(n-1)}]$. Since $\wt g_i(0)\ge \sigma^2h$ in $U$, we see that $\a>0$.   Suppose $\a<\gamma_1$. By continuity, there is $\e>0$ such that $\a+\e<\gamma_1$. Then $\gamma_2 \a^{8(n-1)} \le \gamma_2  <\tau$. By the claim, we can conclude that
$$
 \wt g_i(t)\ge \gamma_1^2 h\geq (\a+\e)^2h
$$
in  $B_{\wt g_i(t)}(p,\sigma)$, $t\in [0,\gamma_2\a^{8(n-1)}]$. By choosing a possibly smaller $\e$ and by continuity, the above inequality is also true for $t\in [0,\gamma_2(\a+\e)^{8(n-1)}]$. This is a contradiction.

To prove the claim, let $\gamma_1$ and $\gamma_2>0$ be two constants to be determined later and are independent of $\a$ and $i$. In the following, $C_k$ will denote a positive constant independent of $\a$ and $i$. In the following, for simplicity in notation, we suppress the index $i$ and simply write $\wt g_i$ as $g$.

Suppose $\a\le \gamma_1$ is such that
$$
g(t)\ge \a^2 h
$$
in $  B_{g(t)}(p,\sigma),\;t\in[0,\gamma_2 \a^{8(n-1)}]$.  By Lemma \ref{slma-1}, $\det(g(t))/\det(h)\le C_1$ for some $C_1>1$. Hence we have
\bee
\a^2h\le g(t)\le C_1\a^{-2(n-1)}h
\eee
 on $B_{g(t)}(p,\sigma),\;t\in[0,\gamma_2 \a^{8(n-1)}]$ and hence on $B_h(p,C_1^{-1/2}\a^{n-1}\sigma)\times [0,\gamma_2 \a^{8(n-1)}]$ because $B_h(p,C_1^{-1/2}\a^{n-1}\sigma)\subset B_{g(t)}(p,\sigma)$ for $t\in[0,\gamma_2 \a^{8(n-1)}]$. This can be seen by considering the maximal $h$-geodesic inside $B_t(p,\sigma)$.  Together with the fact that $\frac14h \le h_E\le 4h$ on $B_0(2)$, we conclude that
\be\label{e-alpha-2}
\a_1^2h_E\le g(t)\le  \a_1^{-2}h_E
\ee
on $  B_0(\frac1{2\sqrt{C_1}}\a^{n-1}\sigma)\times[0, \gamma_2\a^{8(n-1)}]$, where $\a_1>0$ is given by
\be\label{e-alpha-1}
\a_1^2=\frac1{4C_1}\a^{2(n-1)}.
\ee
 By \cite[Theorem 1.1]{ShermanWeinkove2012}, we conclude that
\be\label{e-Rm}
|\Rm(g(t))|\le \frac{C_2}{\a_1^8t}
\ee
on $  B_0(\frac\sigma 2\a_1)\times[0, \gamma_2\a^{8(n-1)}]$.   From the proof in \cite{ShermanWeinkove2012}, the constant $C_2$ depends on an upper bound of the existence time but not its precise value. In particular, it is independent of $\a$ here.  By \eqref{e-alpha-2}, we conclude that \eqref{e-Rm} is true on { $ B_{g(t)}(p, \frac\sigma2\a_1^2)$, $t\in[0, \gamma_2\a^{8(n-1)}]$}.

By \cite[Lemma 8.3]{Perelman2003} (see also \cite[Chapter 18, Theorem 18.7]{Chow}), we have:
\be\label{e-distance-1}
\heat (d_t(p,x)+C_3\a_1^{-4}t^\frac12)\ge0
\ee
in the sense of barrier (see the definition in Appendix \ref{s-max}) outside $B_{g(t)}(p,\a_1^4\sqrt t)$, provided
\be\label{e-t-1}
  t^\frac12\le \frac\sigma2\a_1^{-2}.
\ee

Let $\xi\ge0$ be smooth with $\xi=1$ on $[0,\frac 43]$ and is zero outside $[0,2]$, with $\xi'\le 0, |\xi'|^2/\xi+  |\xi''|\le C $. Let

$$
\Phi(x,t)=\xi( \sigma^{-1}\eta(x,t))
$$
where $\eta(x,t)=d_t(p,x)+C_3\a_1^{-4}t^\frac12$. For any $\e>0$, for $t>0$ satisfying \eqref{e-t-1}, if $d_t(p,x)+C_3\a_1^{-4}t^\frac12<\frac43\sigma$, then $\Phi(x,t)=1$ near $x$ and so
\bee
\heat(\log(\Phi+\e))=0.
\eee
If $d_t(p,x)+C_3\a_1^{-4}t^\frac12\ge\frac43\sigma$ and $d_t(p,x)\ge \a_1^4t^\frac12$, then
 in the sense of barrier we have:

\be\label{e-Phi-1}
\begin{split}
& \heat \log (\Phi+\e)\\
 =&   \lf(\frac{\xi'}{\xi} \sigma^{-1}\heat\eta-\frac{\xi''}{\xi} \sigma^{-2}|\nabla\eta|^2+\frac{(\xi')^2}{\xi^2} \sigma^{-2}|\nabla \eta|^2\ri)\\
 \le & C_4(\Phi+\e)^{-1}.
\end{split}
\ee
by the choice of $\xi$ and \eqref{e-distance-1}. Hence there exists $C_5>0$ such that if

\be\label{e-t-2}
t^\frac12\le C_5\a_1^4
\ee
then $t$ also satisfies \eqref{e-t-1} and $C_3\a_1^{-4}t^{1/2}<\frac\sigma 3$. Moreover, $C_5$ can be chosen so that either $d_t(p,x)+C_3\a_1^{-4}t^\frac12<\frac43\sigma$ or $d_t(p,x)+C_3\a_1^{-4}t^\frac12\ge\frac43\sigma$ and $d_t(p,x)\ge \a_1^4t^\frac12$. Hence \eqref{e-Phi-1} is true in the sense of barrier for $t\in (0, C_5^2\a_1^8]$.

Consider the function
$$
F=\log \tr_hg -Lv+m\log (\Phi+\e)
$$
where  $v=(\tau-t)\dot\phi+\phi-f+nt$, $\tau=e^s-1$. Here $L, m>0$ are constants to be chosen later which are independent of $i,\ \a$. Recall that $v$
satisfies

\bee
\heat v=\tr_g (\omega_0-\tau \Ric(\theta_0)+\ddb f) \geq \b \tr_g h.
\eee
and
\bee
\heat \log \tr_h g\le C_6\tr_g h
\eee
by Lemma \ref{l-a-1} with vanishing torsion terms here. Let
\be\label{e-L}
L\beta= C_6+1+\tau^{-1}.
\ee
Note that by the A.M.-G.M. inequality and the definition of $\dot \phi$, we have
\be\label{e-AMGM}
-\dot \phi  \le n\log \tr_gh;\ \ \log \tr_hg\le \dot\phi +(n-1)\log\tr_gh.
\ee
So
\bee
\log \tr_gh\ge \frac{1}{ n(\tau L-1)+(n-1)}(\log \tr_hg-\tau L\dot\phi)
\eee
Then in the sense of barrier
\bee
\begin{split}
\heat F\le & -\tr_gh+m C_4  (\Phi+\e)^{-1}\\
\le &-\exp\lf(C_7 (\log \tr_hg-\tau L\dot\phi)\ri)+m C_4  (\Phi+\e)^{-1}\\
\le &-\exp\lf(C_7F-C_8-C_7m\log(\Phi+\e)\ri)+m C_4  (\Phi+\e)^{-1}\\
=&-(\Phi+\e)^{-1}mC_4\lf[\exp (C_7F-C_8-\log(mC_4))-1\ri]\\
\end{split}
\eee
if $mC_7=1$, where we have used the upper bound of $\dot \phi$ and the bound of $\phi$ in Lemmas \ref{slma-1}. So
\bee
\begin{split}
&\heat (C_7F-C_8-\log(mC_4))\\ \le& -\frac{mC_4C_7}{ \Phi+\e}\lf[\exp (C_7F-C_8-\log(mC_4))-1\ri]\\
\le&0
\end{split}
\eee
in the sense of barrier whenever $C_7F-C_8-\log(mC_4)>0$. Then by the maximum principle Lemma \ref{max}, we conclude that
$$
C_7F-C_8-\log(mC_4)\le\sup_{t=0}\lf(C_7F-C_8-\log(mC_4)\ri).
$$
Let $\e\to0$, using the definition of $\Phi$, the choice of $C_5$ and   the bound of $|\phi|$, we conclude that in $  B_{g(t)}(p,\sigma)$,
\be\label{e-trace-lower}
\log \tr_hg-L(\tau-t)\dot \phi \le C_9
\ee
provided $t\in [0, C_5^2\a_1^8]$. On the other hand, as in \eqref{e-AMGM}, we have
\bee
\begin{split}
\log \tr_gh\le& -\dot \phi+(n-1)\log\tr_hg\\
=&(n-1)\lf(\log\tr_hg-L(\tau-t)\dot \phi\ri)+(n-1)(L(\tau-t)-1)\dot\phi\\
\le&C_{10}
\end{split}
\eee
provided
\be\label{e-t-3}
Lt\le  L\tau-1.
\ee
Here we have used the upper bound of $\dot \phi$ in Lemma \ref{slma-1}.

Hence there is $\gamma_1>0$ independent of $\a$ and $i$ such that if $t$ satisfies \eqref{e-t-2} and \eqref{e-t-3}, then
$$
g_i(t)\ge \gamma_1^2h
$$
on $B_{g_i(t)}(p,\sigma)$. Let $\gamma_2<\tau$ be such that
 $$
 \gamma_2=\min\{C_5^2,L^{-1}(L\tau-1)\}\times (4C_1)^{-4}
$$
where $C_1, C_5$ are the constants in \eqref{e-alpha-1} and \eqref{e-t-2} respectively and $L$ is given by \eqref{e-L}. If $t\in[0,\gamma_2\a^{8(n-1)}]$, then $t$ will satisfy  \eqref{e-t-2}. One can see that the claim this true.

{ By \eqref{e-trace-lower} and Lemma \ref{slma-1}, we conclude that
$$
\wt g_i(t)\le C_{11}h
$$
on $B_{\wt g_i(t)}(p,\sigma)$ for $t\in[0,\gamma_2\a^{8(n-1)}]$. The upper bound in the Lemma follows by choosing a possibly smaller $\gamma_1$.}

\end{proof}

For the case of Chern-Ricci flow, the result is less satisfactory because the property of $d(x,t)$ does not behave as nice as in the \K case. As before, under the assumptions of Theorem \ref{t-instant-complete}, let $g(t)$ be the Chern-Ricci flow $g(t)$ constructed in the theorem. We have the following:
\begin{prop}\label{p-initial-CR}
With the same notation  and assumptions as in Theorem \ref{t-instant-complete}.  Suppose $\tr_{g_0}h=o(\rho)$. Then  $g(t)\rightarrow g_0$ as $t\rightarrow 0$ in $M$.   The convergence is in $C^\infty$ topology and is uniform in compact subsets of $M$.
\end{prop}
Note that $g_0$ may still be complete. But it may not be equivalent to $h$ and { the curvature of $g_0$ may be unbounded}.

As before, $g(t)$  is the limit of solutions $g_i(t)$ of the unnormalized Chern-Ricci flow on $M\times[0, s)$ with initial data $g_0+\e_i h$ with $\e_i\to 0$. Here we may assume $s\leq 1$.   We want to prove the following:

\begin{lma}\label{l-initial-CR-1}   With the same notation  and assumptions as in Proposition \ref{p-initial-CR} and let $S<\tau:=e^s-1$, for any precompact open subset $\Omega$ of $M$, there is $C>0$,
 $$
 C^{-1}h\le  g_i(t)\le Cg
 $$
 for all $i$  in $\Omega\times[0, S]$.
\end{lma}
Suppose the lemma is true, then Proposition \ref{p-initial-CR} will  follows from the local estimates in \cite{ShermanWeinkove2013} for Chern-Ricci flow. To prove the lemma, first we prove the following.

Let $\phi_i$ be the potential for $g_i$.
\begin{sublma}\label{sl-initial-CR-1} Suppose
$$\liminf_{\rho\to\infty}\rho^{-1}\log\frac{\omega_0^n}{\theta_0^n}\ge0.
$$
Then for any $\sigma>0$  (small enough independent of $i$), there is a constant $C>0 $ independent of $i$ such that
\bee
\dot\phi_i\ge -C -\sigma\rho
\eee
on $M\times[0, S]$.
\end{sublma}
\begin{proof} In the following, we will denote $\phi_i$ simply by $\phi$ and $g_i(t)$ simply by $g(t)$ if there is no confusion arisen. Note that $g(t)$ is uniformly equivalent to $h$. Let $\sigma>0$.

Let $F=-(\tau-t)\dot \phi-\phi+f-nt-\sigma \rho $. By \eqref{e-a-1} and \eqref{e-a-2},  for $0\le t\le S$, we have
\bee
\begin{split}
\heat(-(\tau-t)\dot \phi-\phi)  =& (\tau-t)\tr_g\Ric(\theta_0)+\dot\phi-\dot\phi+\tr_g(\ii\ddbar \phi)\\
=&(\tau-t)\tr_g\Ric(\theta_0)+\lf(n+t\tr_g(\Ric(\theta_0))-\tr_g(\theta_0)\ri)\\
=&\tau\tr_g\Ric(\theta_0)+n-\tr_g\theta_0 \\
\end{split}
\eee
Hence by the fact that:
\bee
\omega_0-\tau\Ric(\theta_0)+\ii\ddbar f\ge \beta\theta_0,
\eee
we have
\bee
\begin{split}
\heat F\le &\tau\tr_g\Ric(\theta_0)-\tr_g\theta_0-\Delta f+\sigma \Delta \rho\\
\le& (-\beta+ \sigma C_1)\tr_g\theta_0\\
<&0
\end{split}
\eee
for some constant $C_1$ independent of $\sigma$ and $i$ for  $\sigma$ with $C_1\sigma<\beta$.
Since $F$ is bounded from above, by the maximum principle Lemma \ref{max}, we conclude that
$$
\sup_{M\times[0, S]}F\le \sup_{M\times\{0\}}F.
$$
At $t=0$,
$$
F=-\tau\dot\phi-\sigma\rho+f.
$$
By the assumption, we conclude that $F\le C(\sigma)$ at $t=0$. Hence we have
$$
F\le C(\sigma)
$$
on $M\times[0, S]$. Since $\phi, f$ are bounded, the sublemma follows.

\end{proof}

\begin{sublma}\label{sl-initial-CR-2} With the same notations as in Sublemma \ref{sl-initial-CR-1}. Suppose $\tr_{g_0}h=o(\rho)$. Then
$$
\tr_hg_i\le C\exp(C'\rho)
$$
on $M\times[0, S]$ for some positive constants $C, C'$ independent of $i$.

\end{sublma}

\begin{proof} We will denote $g_i$ by $g$ again  and $\omega_{0}$ to be the \K form of the initial metric $g_i(0)=g_0+\e_ih$.
Note that
\bee
\begin{split}
\heat\phi=&\dot\phi-\Delta \phi\\
=&\dot \phi-(n-\tr_g\omega_{0}+t\tr_g(\Ric(\theta_0)))\\
\ge& \dot \phi-n+ \tr_g\omega_{0}+\frac{t\beta}{\tau}\tr_gh-\frac{t}{\tau}\tr_g\omega_0-\frac t\tau\Delta f\\
\ge &\dot \phi- n-\frac t\tau\Delta f+(1-\frac S\tau)\tr_g\omega_{0}.
\end{split}
\eee
Then we have:
\be\label{e-initial-CR-1}
\heat (\phi+ nt-\frac t\tau f)\ge    \dot\phi+(1-\frac S\tau)\tr_g\omega_{0}-C_0.
\ee
Since $|\phi|$ is bounded by a constant independent of $i$ on $M\times[0, S]$, see Lemma \ref{l-uudot-upper-1} and Lemma \ref{l-all-u}, there is a constant $C_1, C_2>0$ so that $\xi:=\phi +nt-\frac t\tau f+C_1\ge 1$ and $\xi\le C_2$ on $M\times[0, S]$. Here and below $C_j$ will denote positive constants independent of $i$.
Let $\Phi( \varsigma)=2-e^{-\varsigma}$ for $\varsigma\in \R$. Then for $\xi:=\phi +nt-\frac t\tau f+C_1\ge 1$, we have
\be\label{e-initia-CR-2}
\left\{
  \begin{array}{ll}
    \Phi(\xi)\ge  & 1 \\
   \Phi'(\xi)\ge   & e^{-C_2}\\
   \Phi''(\xi)\le &-e^{-C_2}
  \end{array}
\right.
\ee
on $M\times[0, S]$. Next, let $P(\varsigma)$ be a positive function on $\R$ so that $P'>0$.
Define
$$
F(x,t)=\Phi(\xi)P(\rho).
$$
Let $\Upsilon=tr_hg$, here $g=g_i$. Let $F\to\infty$ near infinity be a smooth function of $x, t$. Then by Lemma \ref{l-a-1}, we have
\bee
  \heat (\log \Upsilon-F)=\mathrm{I+II+III}-\heat F
  \eee
  where

\bee
\begin{split}
\mathrm{I}\le &2\Upsilon^{-2}\text{\bf Re}\lf(  h^{i\bar l} g^{k\bar q}( T_0)_{ki\bar l} \hat \nabla_{\bar q}\Upsilon\ri),
\end{split}
\eee

\bee
\begin{split}
\mathrm{II}=&\Upsilon^{-1} g^\ijb  h^{k\bar l}g_{k\bar q} \lf(\hat \nabla_i \ol{(\hat T)_{jl}^q}- \hat h^{p\bar q}\hat R_{i\bar lp\bar j}\ri),\\
\end{split}
\eee
and

\bee
\begin{split}
\mathrm{III}=&-\Upsilon^{-1} g^{\ijb}  h^{k\bar l}\lf(\hat \nabla_i\lf(\ol{( T_0)_{jl\bar k}} \ri) +\hat \nabla_{\bar l}\lf( {(  T_0)_{ik\bar j} }\ri)-\ol{ (\hat T)_{jl}^q}(  T_0)_{ik\bar q}^p  \ri).
\end{split}
\eee
Let $\Theta=\tr_gh$. Suppose $\log \Upsilon-F$ attains a positive maximum at $(x_0,t_0)$ with $t_0>0$, then at this point,
$$
\Upsilon^{-1}\hat\nabla \Upsilon=\hat\nabla F,
$$
and so
\bee
\begin{split}
\mathrm{I}\le &2\Upsilon^{-2}\text{\bf Re}\lf(  h^{i\bar l} g^{k\bar q}( T_0)_{ki\bar l} \hat \nabla_{\bar q}\Upsilon\ri)\\
\le &C\Upsilon^{-1}\Theta^\frac12|\nabla F|\\
\le&C'\Upsilon^{-1}\Theta^\frac12\lf(P|\nabla \xi|+P'\Theta^\frac12\ri).
\end{split}
\eee
because $|\p\rho|_h$ is bounded.  Here we use the norm with respect to the evolving metric $g(t)$.
%
%
$$
\mathrm{II}\le C\Theta,
$$
$$
\mathrm{III}\le C\Upsilon^{-1}\Theta.
$$
Here $C, C'$ are positive constants independent of $i$. On the other hand,
\bee
\begin{split}
&\heat F\\
=&P\heat\Phi+2{\bf Re}\left(g^\ijb\p_i\Phi\p_{\bar j}P\right)+\Phi\heat P\\
\ge&P\lf(\Phi'\heat \xi -\Phi''|\nabla\xi|^2\ri)-C_4\Phi'P'\Theta^\frac12|\nabla\xi|-C_4\Theta\Phi (P'+|P''|)\\
\ge&P\Phi'\dot\phi+e^{-C_2}P(1-\frac S\tau)\tr_g\omega_0-C_0P+e^{-C_2}P|\nabla\xi|^2-\frac12e^{-C_2}P|\nabla\xi|^2
\\
&-C_5\frac{(P')^2}{P}\Theta-C_4\Theta(P'+|P''|).
\end{split}
\eee
Here we have used the fact that $|\p\rho|_h, |\ddbar\rho|_h$ are bounded  $\Phi(\xi)\le 2$ and \eqref{e-initia-CR-2}.

So at  $(x_0,t_0)$,
\bee
\begin{split}
 &\heat (\log \Upsilon-F)\\
 \le& C_3\lf( \Upsilon^{-1}\Theta^\frac12\lf(P|\nabla \xi|+P'\Theta^\frac12\ri) + \Upsilon^{-1}\Theta+ \Theta\ri)\\
 &- P\Phi'\dot\phi-e^{-C_2}P(1-\frac S\tau)\tr_g\omega_0+C_0P-\frac12e^{-C_2}P|\nabla\xi|^2\\
 &+\Theta\lf(C_5\frac{(P')^2}{P}+
 C_4(P'+|P''|)\ri)\\
 \le&- P\Phi'\dot\phi-e^{-C_2}P(1-\frac S\tau)\tr_g\omega_0+C_0P+\lf(-\frac12e^{-C_2} +\Upsilon^{-1}\ri)P|\nabla\xi|^2
\\
&+C_6\Theta\lf(\Upsilon^{-1}+1+\Upsilon^{-1}P'+P'+\Upsilon^{-1}P+\frac{(P')^2}{p}+|P''|\ri).
 \end{split}
\eee
Now
\bee
-\dot\phi\le c(n)\log \Theta.
\eee
Suppose $\omega_0\ge \frac1{Q(\rho)}\theta_0$ with $Q>0$ and suppose $\Upsilon^{-1}\le \frac12e^{-C_2}$ at $(x_0,t_0)$, then at $(x_0,t_0)$, we have
\bee
\begin{split}
&\heat (\log \Upsilon-F)\\
\le&C_7P(\log\Theta+1)+\Theta\lf[-C_8PQ^{-1}+C_9 \lf( 1+ P'+\frac{(P')^2}{P}+|P''|\ri)\ri].
\end{split}
\eee
By the assumption on $\tr_{g_0}h$, for any $\sigma>0$ there is $\rho_0>0$ such that if $\rho\ge \rho_0$, then $\tr_{g_0}h\le \sigma\rho.$ Hence we can find $C=C(\sigma)$ such that
$$
g_0\ge \frac{1}{\sigma(\rho+C(\sigma))}h
$$
and $\rho+C(\sigma)\ge 1$
on $M$. Let $Q(\rho)=\sigma(\rho+C(\sigma)), P(\rho)=\rho+C(\sigma)$, then above inequality becomes
\bee
\begin{split}
 \heat (\log \Upsilon-F)
\le&C_7P\log (e\Theta)+\Theta\lf(-C_8\sigma^{-1}+3C_9\ri)\\
\le& C_7P\log (e\Theta)-\frac12C_8\Theta
\end{split}
\eee
if we choose $\sigma$ small enough independent of $i$. Since  $\log\Upsilon-F\to-\infty$ near infinity and uniform in $t\in [0, S]$, and $\log\Upsilon-F<0$ at $t=0$, by maximum principle, either $\log\Upsilon-F\le 0$ on $M\times[0, S]$ or there is $t_0>0$, $x_0\in M$ such that
$\log\Upsilon-F$ attains a positive maximum at $(x_0,t_0)$. Suppose at this point $\Upsilon^{-1}\ge\frac12 e^{-C_2}$, then
$$
\log\Upsilon-F\le C_{10}.
$$
Otherwise, at $(x_0,t_0)$ we have
\bee
0\le C_7P\log (e\Theta)-\frac12C_8\Theta.
\eee
Hence we have at this point $\Theta\le C_{11}$ which implies $\Upsilon\le C_{12}$ because $\dot\phi\le C$ for some constant independent of $i$. So
$$
\log\Upsilon-F\le\log C_{12}.
$$
Or
$$
\Theta\le C_{13}P^2.
$$
This implies $\log \Upsilon\le C_{14}(1+\log P)$. Hence
\bee
\log \Upsilon-F\le C_{14}.
\eee
From these considerations, we conclude that the sublemma is true.
\end{proof}
\begin{proof}[Proof of Lemma \ref{l-initial-CR-1}] The lemma follows from Sublemmas \ref{sl-initial-CR-1} and \ref{sl-initial-CR-2}.
\end{proof}

\section{Long time behaviour and convergence}

In this section, we will first study the longtime behaviour for the solution constructed in Theorem \ref{t-instant-complete}. Namely, we will show the following theorem:
\begin{thm}\label{longtime}
Under the assumption of Theorem \ref{t-instant-complete}, if in addition, $$-\Ric(h)+\ddb f\geq \beta \theta_0$$ for some $f\in C^\infty(M)\cap L^\infty(M)$, $\beta>0$. Then the solution constructed from Theorem \ref{t-instant-complete} is a longtime solution and converges to a unique complete \K Einstein metric with negative scalar curvature on $M$.
\end{thm}

Before we prove Theorem \ref{longtime}, let us prove a lower bound of $\dot u$ which will be used in the argument of convergence. Once we have uniform equivalence of metrics, we can obtain a better lower bound of $\dot{u}$.

\begin{lma}\label{du-convergence} Assume the solution constructed from Theorem \ref{t-instant-complete} is a longtime solution, then   there is a positive constant $C$   such that  \bee
\dot{u}\geq-Ce^{-\frac t2} \eee on $M\times[2, \infty)$.
\begin{proof} Since we do not have upper bound of $g(t)$ as $t\to 0$, we shift the initial time of the flow to $t=1$. Note that \bee\begin{split}
\heat (e^t\dot{u}-f)=&-tr_{ g}(\Ric(h)+g(1))+\Delta f\\
\geq&-tr_{ g}g(1)\geq-C_1. \end{split}\eee

Consider $Q=e^t\dot{u}-f+(C_1+1)t$. Then we can use maximum principle argument as before to obtain $Q(x, t)\geq \inf\limits_MQ(0)$. Then we have\bee
e^t\dot{u}\geq -C_2-(C_1+1)t \eee which implies \bee
\dot{u}\geq -C_3e^{-\frac t2} \eee on $M\times[1, \infty)$. We shift the time back, we obtain the result.

\end{proof}
\end{lma}

\begin{proof}[Proof of Theorem \ref{longtime}] The assumption $-\Ric(h)+\ddb f\geq \beta \theta_0$ implies that for all $s$ large enough, \bee
-Ric(h)+e^{-s}(\omega_0+Ric(h))+\ddb \hat f\geq \frac\beta2 \theta_0. \eee Here $\hat f=(1-e^{-s})f$ is a bounded function on $M$. By Theorem \ref{t-instant-complete} and Lemma \ref{l-trace-2},  \eqref{e-MP-1} has a smooth solution on $M\times(0, \infty)$ with $g(t)$ uniformly equivalent to $h$ on any $[a, \infty)\subset(0, \infty)$. Combining the local higher order estimate of Chern-Ricci flow (See \cite{ShermanWeinkove2013} for example) and Lemma \ref{du-convergence}, we can conclude that $u(t)$ converges smoothly and locally to a smooth function $u_\infty$ as $t\to\infty$ and $\log\frac{\omega^n_\infty}{\theta_0^n}=u_\infty$. Taking $\ddb$ on both sides, we have \bee
-\Ric(g_\infty)+\Ric(h)=\ddb u_\infty, \eee which implies $-\Ric(g_\infty)=g_\infty$. Obviously, $g_\infty$ is K\"ahler.  Uniqueness follows from \cite[Theorem 3]{Yau1978} (see also Proposition 5.1 in \cite{HLTT}).

\end{proof}

Taking $g_0=h$ in the theorem, we have

\begin{cor} Let $(M,h)$ be a complete Hermitian manifold satisfying the assumptions in Theorem \ref{longtime}. Then the Chern-Ricci flow with initial data $h$ exists on $M\times[0,\infty)$ and converge uniformly on any compact subsets to the unique complete K\"ahler-Einstein metric with negative scalar curvature on $M$.
\end{cor}

For \KR flow, we have the following general phenomena related to Theorem \ref{longtime}.
\begin{thm}\label{convergence-krf}
Let $(M,h)$ be a smooth complete Hermitian manifold with
$\mathrm{BK}(h) \geq -K_0$ and $|\nabla^h_{\bar\partial}T_h|_h\leq K_0$ for some constant $K_0\geq 0$. Moreover, assume \bee -\Ric(h)+\ddb f \geq k h \eee
for some constant $k>0$ and function $f\in C^\infty(M)\cap L^\infty(M)$. Suppose $g(t)$ is a smooth complete solution to the normalized \KR flow on $M\times[0,+\infty)$ with $g(0)=g_0$ which satisfies  \bee
\frac{\det g_0}{\det h}\leq \Lambda \eee and \bee R(g_0)\geq -L \eee for some $\Lambda,L>0$. Then $g(t)$ satisfies \bee
C^{-1}h\leq  g(t)\leq C h \eee on $M\times[1, \infty)$ for some constant $C=C(n, K_0, k, ||f||_\infty, \Lambda, L)>0$. In particular, $ g(t)$ converges to the unique smooth complete K\"ahler-Einstein metric with negative scalar curvature.
\end{thm}

\begin{proof} We can assume $k=1$, otherwise we rescale $h$. We consider the corresponding unnormalized K\"ahler-Ricci flow $\wt g(s)=e^{t}g(t)$ with $s=e^t-1$. Then the  corresponding Monge-Amp\`ere equation to the unnormalized  K\"ahler-Ricci flow is: \bee\left\{\begin{array}{l}
      \frac{\p}{\p s}\phi=\displaystyle{\log\frac{(\omega_0-s\Ric(\theta_0)+\ddb\phi)^n}{\theta_0^n}}  \\
        \phi(0)=0.
\end{array} \right. \eee

Here $\theta_0$ is the \K form of $h$. By the assumption $R(g_0)\geq -L $, Proposition 2.1 in \cite{Chen2009} and Lemma 5.1 in \cite{HLTT} with the fact \bee
\sheat \wt R\geq \frac{1}{n}\wt R^2,\eee we conclude that $\wt R:=R(\wt g(s))\geq \max\{-L, -\frac ns\}$ on $M\times[0, \infty)$. Note that $\ddot{\phi}=-R(\wt g(s))$, we have on $M\times[0, 1]$, $\dot{\phi}\leq C(L, \Lambda)$; on $M\times[1, \infty)$, $\dot{\phi}\leq C(L, \Lambda)+n\log s$.

For lower bound of $\dot{\phi}$, we consider $Q=-\dot{\phi}+f$. We compute: \bee\begin{split}
\sheat Q=&-\sheat \dot{\phi}-\Delta f \\
=&tr_{\wt g}[\Ric(\theta_0)-\ddb f]\\
\leq&-tr_{\wt g}h\\
\leq&-ne^{-\frac{\dot{\phi}}{n}}\\
\leq&-ne^{\frac{1}{n}(Q-f)}\\
\leq&-C(n, ||f||_\infty)e^{\frac{Q}{n}}\\
\leq&-C(n, ||f||_\infty)Q^2, \end{split}\eee whenever $Q>0$.

Then by the same argument as in the proof of Proposition 2.1 in \cite{Chen2009}, we conclude that $\dot{\phi}\geq -C(n, \lambda, ||f||_\infty)$ on $M\times[0, \infty)$. Here $\lambda$ is the lower bound of $\frac{\det g_0}{\det h}$. However, this estimate is not enough for later applications. We consider $F=-\dot{\phi}+f+n\log s$. Then we similarly obtain \bee
\sheat F\leq -C(n, ||f||_\infty)F^2, \eee whenever $F>0$. By Lemma 5.1 in \cite{HLTT}, we conclude that $F\leq\frac{C(n, ||f||_\infty)}{s}$ on $M\times[0, \infty)$. Therefore, we obtain \bee
\dot{\phi}\geq -C(n, ||f||_\infty)+n\log s \eee on $M\times[1, \infty)$.

To sum up, for the bound of  $\dot{\phi}$, we have:

(i) On $M\times[0, 1]$, $-C(n, \lambda, ||f||_\infty)\leq \dot{\phi}\leq C(L, \Lambda)$;

(ii) On $M\times[1, \infty)$, $-C(n, ||f||_\infty)+n\log s\leq \dot{\phi}\leq C(L, \Lambda)+n\log s$.

Then we consider back to the normalized K\"ahler-Ricci flow $g(t)$. Since \bee
\log\frac{\det g(t)}{\det h}=-n\log (s+1)+\frac{\p}{\p s}\phi(s), \eee where $s=e^t-1$, we obtain: \bee
-C(n, ||f||_\infty)\leq\dot{u}(t)+u(t)\leq C(L, \Lambda) \eee on $M\times[\log 2, \infty)$. Here $u$ solves \eqref{e-MP-1}.

Next, we consider $G(x, t)=\log tr_h g(t)-A(\dot{u}(t)+u(t)+f)$. Here $A$ is a large constant to be chosen. As in Section 1, we have \bee
\heat \log tr_h g(t)\leq C(n, K_0)tr_{ g(t)}h-1. \eee Therefore, \bee\begin{split}
\heat G\leq& C(n, K_0)tr_{ g(t)}h-1+An+A(tr_{ g}\Ric(h)+tr_{ g}\ddb f)\\
\leq& (-A+C(n, K_0))tr_{g(t)}h-1+An\\
\leq& -tr_{g(t)}h+An. \end{split}\eee Here we take $A=C(n, K_0)+1$.

On the other hand, \bee tr_h g(t)\leq\frac{1}{(n-1)!}\cdot(tr_{ g(t)}h)^{n-1}\cdot\frac{\det g}{\det h}\leq C(n, L, \Lambda)(tr_{ g(t)}h)^{n-1}. \eee

Then we have \bee\begin{split}
\heat G\leq&-C(n, L, \Lambda)(tr_h g(t))^{\frac{1}{n-1}}+C(n,K_0)\\
=&-C(n, L, \Lambda)e^{\frac{1}{n-1}\log tr_hg(t)}+C(n,K_0)\\
=&-C(n, L, \Lambda)e^{\frac{1}{n-1}[G+A(\dot{u}(t)+u(t)+f)]}+C(n,K_0)\\
\leq&-C(n, L, \Lambda, ||f||_\infty)e^{\frac{1}{n-1}G}+C(n,K_0)\\
\leq&-C(n, L, \Lambda, ||f||_\infty)G^2+C(n,K_0), \end{split}\eee whenever $G>0$.

By similar argument as in the proof of Lemma 5.1 in \cite{HLTT}, we conclude that $G\leq C(n, L, \Lambda, ||f||_\infty, K_0)$ on $M\times[1, \infty)$. The difference here is that we consider the normalized K\"ahler-Ricci flow instead of K\"ahler-Ricci flow. The Perelman's distance distortion lemma for normalized K\"ahler-Ricci flow is the following:\bee
\heat d_t(x_0, x)\geq -\frac{5(n-1)}{3}r_0^{-1}-d_t(x_0, x). \eee We then consider $t\cdot\phi(\frac{1}{Ar_0}[e^t\cdot d_t(x_0, x)+\frac{5(n-1)e^t}{3}r_0^{-1}])\cdot G(x,t)$, the results follows from the same argument as in the proof of Lemma 5.1 in \cite{HLTT}.

This implies \bee
 g(t)\leq C(n, L, \Lambda, ||f||_\infty, K_0)h \eee  on $M\times[1, \infty)$.

For lower bound, combining with $e^{\dot{u}(t)+u(t)}=\frac{\det  g}{\det h}$, we have \bee
 g(t)\geq C^{-1}(n, L, \Lambda, ||f||_\infty, K_0)h \eee  on $M\times[1, \infty)$.

Once we obtain the uniform equivalence of metrics of the normalized K\"ahler-Ricci flow, the convergence follows from the same argument as in the proof of Theorem 5.1 in \cite{HLTT}. This completes the proof of Theorem \ref{convergence-krf}.

\end{proof}

\appendix
\section{Some basic relations}

Let $g(t)$ be a solution to the Chern-Ricci flow,
 $$
  \partial_tg=-\Ric(g)
 $$
 and $h$ is another Hermitian metric. Let $\omega(t)$ be the \K form of $g(t)$, $\theta_0$ be the \K form of $h$. Let
$$
\phi(t)=\int_0^t\log \frac{\omega^n(s)}{\theta_0^n}ds.
$$

\be\label{e-a-1}
\omega(t)=\omega(0)-t\Ric(\theta_0)+\ii\ddbar\phi.
\ee

Let $\dot\phi=\frac{\p}{\p t}\phi$. Then
\be\label{e-a-2}
\heat\dot\phi=-\tr_g(\Ric(\theta_0)),
\ee
where $ \Delta$ is the Chern Laplacian with respect to $ g$.

On the other hand, if $g$ is as above, the solution $\wt g$ of the corresponding  normalized Chern-Ricci flow with the same initial data
$$
\partial_t\wt g=-\Ric(\wt g)-\wt g
$$
is given by
$$
\wt g(x,t)=e^{-t}g(x,e^{t}-1).
$$
The corresponding potential $u$ is given by
$$
u(t)=e^{-t}\int_0^te^s\log \frac{\wt \omega^n(s)}{\theta_0^n}ds
$$
where $\wt \omega(s)$ is the \K form of $\wt g(s)$. Also,
\be\label{e-a-3}
\wt\omega(t)=-\Ric(\theta_0)+e^{-t}(\Ric(\theta_0)+\omega(0))+\ii\ddbar u.
\ee

\be\label{e-a-5}
\lf(\frac{\p}{\p t}-\wt \Delta\ri)(\dot u+u)=-\tr_{\wt g}\Ric(\theta_0)-n,
\ee
where $\wt \Delta$ is the Chern Laplacian with respect to $\wt g$.

\begin{lma}[See \cite{TosattiWeinkove2015,Lee-Tam}]\label{l-a-1} Let $g(t)$ be a solution to the Chern-Ricci flow and  let $\Upsilon=\tr_{  h}g$, and $\Theta=\tr_g  h$.

  \bee
  \heat \log \Upsilon=\mathrm{I+II+III}
  \eee
  where

\bee
\begin{split}
\mathrm{I}\le &2\Upsilon^{-2}\text{\bf Re}\lf(  h^{i\bar l} g^{k\bar q}( T_0)_{ki\bar l} \hat \nabla_{\bar q}\Upsilon\ri).
\end{split}
\eee
\bee
\begin{split}
\mathrm{II}=&\Upsilon^{-1} g^\ijb \hat h^{k\bar l}g_{k\bar q} \lf(\hat \nabla_i \ol{(\hat T)_{jl}^q}- \hat h^{p\bar q}\hat R_{i\bar lp\bar j}\ri)\\
\end{split}
\eee
and

\bee
\begin{split}
\mathrm{III}=&-\Upsilon^{-1} g^{\ijb}  h^{k\bar l}\lf(\hat \nabla_i\lf(\ol{( T_0)_{jl\bar k}} \ri) +\hat \nabla_{\bar l}\lf( {(  T_0)_{ik\bar j} }\ri)-\ol{ (\hat T)_{jl}^q}(  T_0)_{ik\bar q}^p  \ri).
\end{split}
\eee
where $T_0$ is the torsion of $g_0=g(0)$, $\hat T$ is the torsion of $h$ and $\hat\nabla$ is the derivative with respect to the Chern connection of $h$.

\end{lma}

\section{A maximum principle}\label{s-max}

We have the following   maximum principle, see \cite{HLTT} for example.

\begin{lma}\label{max} Let $(M^n,h)$ be a complete non-compact Hermitian manifold satisfying condition:  There exists a  smooth positive real exhaustion function $\rho$  such that $|\partial \rho|^2_h+|\sqrt{-1}\partial\bar\partial \rho|_h\leq C_1$. Suppose  $g(t)$ is a solution to the Chern-Ricci flow on $M\times[0,S)$. Assume for any $0<S_1<S$, there is $C_2>0$ such that
$$
C_2^{-1}h\le g(t)
$$
for $0\leq t\le S_1$.
Let $f$ be  a smooth function    on $M\times[0,S)$ which is bounded from above such that
$$
\heat f\le0
$$
on $\{f>0\}$ in the sense of barrier. Suppose  $f\le 0$ at $t=0$, then $f\le 0$ on $M\times[0,S)$.
\end{lma}
{We say that
$$
\heat f\le \phi
$$
 in the sense of barrier means that for fixed $t_1>0$ and $x_1$,    for any $\e>0$, there is a smooth function $\sigma(x)$ near $x$ such that $\sigma(x_1)=f(x_1,t_1)$, $\sigma(x)\le f(x,t_1)$ near $x_1$, such that $\sigma$ is $C^2$ and at $(x_1,t_1)$
\bee
\frac{\p_-}{\p t}f(x,t)-\Delta  \sigma(x)\le \phi(x)+\e.
\eee
Here
\bee
\frac{\p_-}{\p t}f(x,t)=\liminf_{h\to 0^+}\frac{f(x,t)-f(x,t-h)}h.
\eee
for a function $f(x,t)$.}

\end{document}